\documentclass[11pt, a4paper, oneside, reqno]{amsart}

\usepackage[margin=1in]{geometry}
\usepackage{amsfonts}
\usepackage{amsmath}
\usepackage{amssymb,bbm}
\usepackage{amsthm}
\usepackage{hyperref}
\usepackage{mathtools}
\usepackage{tikz}
\usepackage{tikz-cd}
\usepackage{enumerate}
\usepackage{stmaryrd} 
\usepackage{mathabx}  
\usepackage{xcolor}
\usepackage{tipa}
\makeatletter
\usepackage{svg}
\renewcommand{\tocsection}[3]{%
	\indentlabel{\@ifnotempty{#2}{\bfseries\ignorespaces#1 #2.\,\,}}\bfseries#3}
\renewcommand{\tocsubsection}[3]{%
	\indentlabel{\@ifnotempty{#2}{\ignorespaces#1 #2\quad}}#3}
\renewcommand{\tocsubsubsection}[3]{%
	\quad\quad\quad\indentlabel{\@ifnotempty{#2}{\ignorespaces#1 #2\quad}}#3}

\def\1{\mathbf{1}}

\newcommand{\sK}{{\sf K}}
\newcommand{\copairpq}{{[{\bf p},{\bf q}]}}
\newcommand{\setK}{{D(I,J)}}
\newcommand{\partitionskl}{{D(k,l)}}
\newcommand{\sM}{{\sf M}}

\newtheorem{theorem}{Theorem}
\newtheorem{proposition}[theorem]{Proposition}
\newtheorem*{proposition*}{Proposition}

\newtheorem*{theorem*}{Theorem}
\newtheorem{lemma}[theorem]{Lemma}

\theoremstyle{remark}
\newtheorem{remark}[theorem]{Remark}
\newtheorem{example}[theorem]{Example}

\theoremstyle{definition}
\newtheorem{definition}[theorem]{Definition}

\numberwithin{equation}{section}
\title{Second moments associated with finite free multiplicative convolution}
\author{Nicolas Gilliers}
\address{Université Paris Cité, CNRS, MAP5, F-75006 Paris, France}
\email{nicolas.gilliers@u-paris.fr} 

\author{Andrea Pettenello}
\address{Université Paris Cité, CNRS, MAP5, F-75006 Paris, France}
\email{andrea.pettenello@etu.u-paris.fr}

\begin{document}
\begin{abstract}
We compute the second mixed moment of elementary symmetric polynomials in the eigenvalues of the random matrix model $AUBU^*$, where $A$ and $B$ are positive, and $U$ is a Haar-distributed random matrix, associated with finite free multiplicative convolution.
\end{abstract}
\maketitle
 \tableofcontents

\section{Introduction}
\subsection{Motivation}
\label{sec:Intro}
Finite free probability was first introduced by Marcus, Spielman and Srivastava in the seminal paper \cite{marcus2022finite}. Its main object of study is the expected characteristic polynomial of the sum or product of randomly rotated matrices. These expected characteristic polynomials are related to certain convolution operations on polynomials first studied over a century ago by Walsh \cite{walsh1922location} and Szeg\"{o} \cite{szego1922bemerkungen}. Specifically, if $A$ and $B$ are $d \times d$ Hermitian matrices with respective characteristic polynomials $p$ and $q$, and $U$ is a Haar-distributed unitary matrix, then the expected characteristic polynomials of $A+UBU^*$ and $AUBU^*$ are respectively equal to the symmetric additive convolution $\boxplus_d$ and the symmetric multiplicative convolution $\boxtimes_d$ of the polynomials $p$ and $q$. Remarkably, these operations are connected to the free additive convolution $\boxplus$ and the free multiplicative convolution $\boxtimes$ of Voiculescu \cite{voiculescu1991limit}. In this sense, finite free probability can be seen as a finite-dimensional approximation of Voiculescu's free probability \cite{voiculescu1986addition, voiculescu1987multiplication}.
\par
The study of the expectation of the characteristic polynomial of a matrix is equivalent to the study of the expectation of the elementary symmetric polynomials in the eigenvalues of that matrix, since the latter are precisely the coefficients of the characteristic polynomial. In the conclusion of his thesis \cite{mirabelli2021hermitian}, Mirabelli opened up the question of finding explicit formulas for the second moments of these elementary symmetric polynomials. Such formulas would provide more information on the distribution of the eigenvalues of randomly rotated matrices, as they would allow us to understand how the eigenvalues fluctuate around the mean. Furthermore, they would be a first, albeit small, step in the development of a theory of second-order finite free probability which would hopefully approximate the second-order free probability of \cite{mingo2006second, mingo2007second, collins2007second} similar to how finite free probability approximates free probability as described above. In this paper, we partially answer Mirabelli's question by determining an explicit formula for the expectation of the product of the $k$-th and $l$-th elementary symmetric polynomials in the eigenvalues of the random matrix $AUBU^*$ (our method is also valid for the sum $A+UBU^*$ but the final formula is more complicated and less easily interpretable, so we chose to omit it for now). The resulting formula expresses the second moment as a sum over certain matchings weighted by weak ballot numbers and monomial symmetric polynomials in the eigenvalues of $A$ and $B$ (see Section \ref{sec:MainResult}).
\par 
Our approach is similar to that of Campbell and Yin in \cite{campbell2021finite}, where they calculate the first moment of the elementary symmetric polynomials, in that we use Weingarten calculus (see Section \ref{sec:WeinCalc}) to reduce the calculation of the second moment to a combinatorial problem over the symmetric group. For the first moment, the Weingarten expansion leads to an average of characters over the whole symmetric group; due to orthogonality of characters, only the trivial character contributes. For the second moment, the average is over the subgroup $S_k \times S_l$ of the symmetric group $S_{k+l}$. Because the pair $(S_{k+l},S_{k}\times S_l)$ is a Gelfand pair (see Section \ref{sec:HahnPoly}), the averaged characters, which are the spherical functions of $(S_{k+l},S_k\times S_l)$, vanish for all but a handful of partitions, specifically those with at most two rows. In both cases, these cancellations reduce the sum to a tractable form.
We remark that the Gelfand pair structure that we exploit for our calculations does not generalize to the calculation of higher moments or the replacement of the unitary group by the orthogonal group; in such cases, Gelfand pairs do not appear, thus making the calculation apparently much harder. Overcoming these obstacles would be an interesting direction for further research.

\subsection{Main Result}
\label{sec:MainResult}
We now set up notation to state our main result more precisely.
\begin{definition}
	\label{def:cycleszizgzag} Let $d \geq 1$ be a positive integer and $I,J \subseteq [d]$ be two subsets.
	We denote $\setK$ the set of partitions of $I \sqcup J$ composed of blocks having at most one element in $I \subset I \sqcup J$ and at most one element in $J \subset I \sqcup J$.
	We denote $\partitionskl$ the set $D(I,J)$ with $I = \{1,...,k\}$ and $J = \{k+1,...,k+l\}$).
\end{definition}
Let $\sK_1, \sK_2 \in \setK$. We consider two different types of blocks of the partition $\sK_1 \vee \sK_2$:
\emph{alternating cycles} and \emph{alternating segments}. They are defined as follows.
\begin{enumerate}
    \item An \emph{alternating cycle} is a block of $\sK_1 \vee \sK_2$ of cardinality greater than or equal to 2 containing no singletons of either $\sK_1$ or $\sK_2$.

    \item  An \emph{alternating segment} is a block of $\sK_1 \vee \sK_2$ containing
      exactly two elements which are singletons in $\sK_1$, or exactly two elements
      which are singletons in $\sK_2$.

\end{enumerate}

\begin{remark}
Let us briefly explain the choice of terminology for alternating cycles and segments; as we shall see, it is quite self-explanatory. Let $G(\sK_1, \sK_2)$ be the graph whose vertices are the elements of $I\sqcup J$ and whose edges are the 2-element blocks of $\sK_1$ (colored \emph{red}) and $\sK_2$ (colored \emph{blue}). Then alternating cycles of $\sK_1 \vee \sK_2$ correspond to connected components of $G(\sK_1, \sK_2)$ that are cycles such that the edges alternate between red and blue and alternating segments are connected components of $G(\sK_1, \sK_2)$ that are segments whose initial and terminal vertices are singletons in either $\sK_1$ or $\sK_2$ and whose edges alternate between red and blue.
\end{remark}
\begin{example}
\label{ex:ExampleCyclesSegments}
    Let $I = \{1,2,5,6,9\}$ and $J = \{1,4,5,6,7,10\}$. Let 
    $$\sK_1 = \{\{4\},\{1_I\},\{1_J\},\{2,5_J\},\{5_I,7\}, \{6_I,6_J\},\{9,10\}\}
    $$ and 
    $$\sK_2 = \{\{1_J\}, \{5_I\}, \{7\}, \{1_I,5_J\},\{2,4\},\{9,6_J\},\{6_I,10\}\},$$
    where the notation $x_I$ (resp. $x_J$) is used to indicate that we are viewing the element $x \in I\cap J$ as belonging to $I \subset I \sqcup J$ (resp. $J \subset I \sqcup J)$.
    
    Then the join $\sK_1 \vee \sK_2$ is
    $$\sK_1 \vee \sK_2 = \{\{6_I,6_J,9,10\},\{1_I,2,4,5_J\}, \{5_I,7\}, \{1_J\}\}.
    $$
    The block $\{6_I, 10, 9, 6_J\}$ is an alternating cycle and the blocks $\{1_I,2,4,5_J\}$ and $\{5_I,7\}$ are alternating segments. In the graph $G(\sK_1, \sK_2)$, the connected components corresponding to these blocks are illustrated in Figure \ref{fig:CyclesAndSegments} below.
\begin{figure}[h]
    \centering
      \def\svgwidth{0.5\linewidth}
     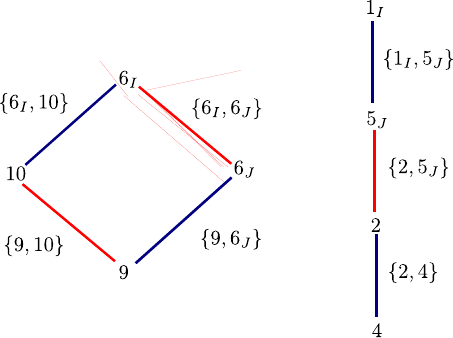
    \caption{On the left is the connected component in $G(\sK_1, \sK_2)$ corresponding to the alternating cycle of $\sK_1 \vee \sK_2$ in Example \ref{ex:ExampleCyclesSegments} and on the right is the connected component corresponding to one of the two alternating segment.}
    \label{fig:CyclesAndSegments}
\end{figure}
\end{example}

We denote ${\sf C}({\sf K}_1, {\sf K}_2)$ the number of alternating cycles of $\sK_1 \vee \sK_2 $ and ${\sf S}({\sf K}_1, {\sf K}_2)$ the number of alternating segments of $\sK_1 \vee \sK_2$. If $\sK_1$ and $\sK_2$ have types $t(\sK_1) = (2^{y_1},1^{k+l-2y_1})$ and $t(\sK_2) = (2^{y_2},1^{k+l-2y_2})$, we define 
$$
	\tilde{\sf S}(\sK_1, \sK_2) \vcentcolon= \frac{1}{2}(y_1+y_2+{\sf S}(\sK_1, \sK_2)).
$$
\newcommand{\tildeS}[2]{\tilde{\sf S}(#1,#2)}

For partitions $\sK_1, \sK_2 \in \partitionskl$, we use the following notation for the number of shared singletons in the partitions ${\sf K}_1$ and ${\sf K}_2$
\begin{align*}
	 & (\sK_1, \sK_2 )[k+l]  &&= |\{ x \in [k+l]:\{x\} \in \sK_1\cap \sK_2\}| && = l+k - 2\tildeS{\sK_1}{\sK_2}, \\
	 & (\sK_1,\sK_2 )[k]  &&= |\{ x \in [k]:\{x\} \in \sK_1\cap \sK_2\}| &&= k - \tildeS{\sK_1}{\sK_2}.
\end{align*}
The right-most equalities in the above formulae are consequences of equations \eqref{eqn:relationy} and \eqref{eqn:relationtofixedpoints} in section \ref{prop:FormulePourSx}. We are now ready to state our main result.
\begin{theorem}\label{thm:MainTheorem}
Let $d\geq 1$ be a positive integer and $A,B \in \mathcal{M}_{d,d}(\mathbb{C})$ two complex Hermitian matrices with positive eigenvalues. Denote by $a = (a_1,\ldots,a_d)$, $b=(b_1,\ldots,b_d)$ their eigenvalues. Let $U$ be a Haar-distributed random unitary matrix. Let $1 \leq k\leq l \leq d$ be two integers such that $d \geq k+l$. Then,
\begin{align*}
	 & \mathbb{E}[e_k(AUBU^{\star})e_l(AUBU^{\star})]                      \\
	 & \hspace{1cm}=\frac{1}{l!k!}\sum_{y,m = 0}^{k} \tilde{M}_m(a)\tilde{M}_y(b)  \\
	 & \hspace{0.5cm}\sum_{\substack{\sK_1, \sK_2\in \partitionskl                              \\ {t({\sK}_1)}=(2^m,1^{k+l-2m}) \\ {t({\sf K}_2)}=(2^y,1^{k+l-2y})}}\hspace{-0.5cm}{2^{{\sf C}(\sK_1,\sK_2)}}\sum_{h = 0}^{(\sK_1,\sK_2 )[k]}\tilde{B}((\sK_1,\sK_2)[k+l]-h,h) \frac{1}{(d+1)_{h+ {\tilde{\sf S}}(\sK_1,\sK_2)}(d)_{k+l-h-{\tilde{\sf S}}(\sK_1,\sK_2)}}
\end{align*}
with
\begin{enumerate}
    \item $
\tilde{B}((\sK_1,\sK_2)[k+l]-h,h)=\frac{B((\sK_1,\sK_2)[k+l]-h,h)}{\binom{(\sK_1,\sK_2)[k+l]}{(\sK_1,\sK_2)[k]}}=\frac{\binom{(\sK_1,\sK_2)[k+l]}{h}-\binom{(\sK_1,\sK_2)[k+l]}{h-1}}{\binom{(\sK_1,\sK_2)[k+l]}{(\sK_1,\sK_2)[k]}},
$
where the ${B}(p,q),\, p \geq q $, are the \emph{weak ballot numbers}:
$$
B(p,q) = \binom{p+q}{q}-\binom{p+q}{q-1}.
$$ 
The numbers $B(p,q)$ count the number of lattice paths from  $(0,0)$ to $(p+q,p-q)$ with steps $(1,1)$ and $(1,-1)$ that never dip below the $x$-axis. When $p=q$ this is exactly a Dyck path, and $B(p,p)$ coincides with the $p$-th Catalan number.
\item $\tilde{M}_r(a) = r!(k+l-2r)! M_r(a)$, where $M_r(a)$ is the monomial symmetric polynomial associated to the partition $(2^r, 1^{k+l-2r})$ in the variables $a = (a_1,...,a_d)$.
\item For positive integers $d \geq n \geq 1$, the quantity $(d)_n = d(d-1)...(d-n+1)$ is the \emph{falling factorial} (or \emph{Pochhammer symbol}).

\end{enumerate}

\end{theorem}
As outlined in Section \ref{sec:Intro}, the proof of Theorem \ref{thm:MainTheorem} relies on the Weingarten calculus (see Section \ref{sec:WeinCalc}) to reduce the expectation to a sum over partitions $\lambda$ of $k+l$ of irreducible characters averaged over the Young subgroup $S_k \times S_l$ of the symmetric group $S_{k+l}$. A key step consists in recognizing that the averaged irreducible characters are precisely the spherical functions of the Gelfand pair $(S_{k+l}, S_k \times S_l)$, thus reducing the sum over all partitions to a sum over partitions of a very specific type. We then exploit known formulas for the spherical functions of the Gelfand pair $(S_{k+l}, S_k \times S_l)$ in terms of Hahn polynomials (see Section \ref{sec:HahnPoly}), which finally leads us to a tractable, though quite technical, combinatorial problem.

\section{Background and Notation}
In this section we recall some background material that is essential for the proof of Theorem \ref{thm:MainTheorem}.

\subsection{Weingarten Calculus}
\label{sec:WeinCalc}
Let $U_d$ be the compact group of $d \times d$ complex unitary matrices, which we equip with its  Haar probability measure ${\rm d}U$. We denote ${\bf u}_{ij}: U_d \to \mathbb{C}$ the random variable that maps a unitary matrix $U = (u_{ij})_{1 \leq i,j \leq d}$ to its $(i,j)$-th entry $u_{ij}$.
The {\it Weingarten calculus} is a family of combinatorial techniques to calculate integrals of polynomial functions on $U_d$. Such integrals were first studied by Weingarten in \cite{weingarten1978asymptotic}. His techniques were significantly extended by Collins in \cite{collins2003moments} and Collins and {\'S}niady in \cite{collins2006integration}, who proved the following theorem.
\begin{theorem*}\cite[Corollary 2.4]{collins2006integration}
	Let ${\bf i},{\bf j}:[n]\to [d]$ and ${\bf i'},    {\bf j'}:[n'] \to [d]$ for positive integers $n,n' \geq 1$. Then
	$$
		\int_{U_d} {\bf u}_{{\bf i}(1){\bf j}(1)}...{\bf u}_{{\bf i}(n){\bf j}(n)} \bar{{\bf u}}_{{\bf i'}(1){\bf j'}(1)}...\bar{{\bf u}}_{{\bf i'}(n'){\bf j'}(n')} {\rm d}U = \sum_{\pi, \sigma \in S_n} \delta_{{\bf i},{\bf i'\circ \pi}}\delta_{{\bf j},{\bf j'\circ \sigma}} {\rm Wg}_d(\pi, \sigma)
	$$
	if $n = n'$ and 0 otherwise.
	\newline Here $\delta_{{\bf i},{\bf i'\circ \pi}} = \delta_{{\bf i}(1){\bf i'}(\pi(1))}...\delta_{{\bf i}(n){\bf i'}(\pi(n))}$ and ${\rm Wg}_d: S_n \times S_n  \to \mathbb{C}$ is a class function known as the {\rm Weingarten function}.
\end{theorem*}
In \cite{collins2006integration}, Collins and {\'S}niady also provide an explicit formula for the Weingarten function ${\rm Wg}_d$ in terms of the irreducible  characters of $S_n$.
\begin{proposition*}\cite[Proposition 2.3, point 2]{collins2006integration}
	Let $n,d \geq 1$ be positive integers. For $\pi, \sigma \in S_n$ we have
	$$
		{\rm Wg}_d(\pi,\sigma) = \frac{1}{(n!)^2}\sum_{\substack{\lambda \vdash n\\ l(\lambda) \leq d}} \frac{\chi^{\lambda}({\rm id})^2}{s_\lambda (1^d)} \chi^{\lambda}(\pi ^{-1}\sigma),
	$$
	where $\chi^{\lambda}$ and $s_\lambda$ are respectively  the irreducible character and the Schur function associated to the partition $\lambda \vdash n$.
\end{proposition*}

For $\lambda$ a partition of $k+l$ we write $c(\lambda, d)$ for the coefficient
\begin{equation}
\label{eqn:coeffc}
c(\lambda, d) = \frac{1}{(k+l)!^2}\frac{\chi^{\lambda}({\rm id})^2}{s_\lambda(1^d)},
\end{equation}
where $l(\lambda) \leq k+l$.
\subsection{Symmetric Polynomials} \label{sec:SymmPoly} We refer the reader to the monograph \cite{macdonald1998symmetric}.
We denote an integer partition $\lambda$ either $\lambda = (\lambda_1,...,\lambda_n)$ or $\lambda = (1^{m_1(\lambda)},2^{m_2(\lambda)},...)$, where $m_i(\lambda)$ denotes the multiplicity of $i$ in $\lambda$. We denote $l(\lambda) = \sum_i m_i(\lambda)$ the length of the partition $\lambda$. Recall that the ring of symmetric polynomials $\Lambda_n$ is the sub-ring of fixed polynomials of the  ring $\mathbb{Z}[x_1,...,x_n]$ under the action of $S_n$ that permutes the variables.
If $\lambda$ is a partition with $l(\lambda) \leq n$  the {\it monomial symmetric polynomial } $M_\lambda$ associated to the partition $\lambda$ is the symmetric polynomial
$$
	M_\lambda = \sum_\alpha x^\alpha,
$$
where $\alpha = (\alpha_1,...,\alpha_n)$ ranges through all permutations of $\lambda = (\lambda_1,...,\lambda_n)$ and $x^\alpha = x_1^{\alpha_1}...x_n^{\alpha_n}$ (if $l(\lambda) < n$ we pad the partition $\lambda$ with zeros at the end). Monomial symmetric polynomials naturally appear in Theorem \ref{thm:MainTheorem} in the following way.
\begin{proposition}
\label{prop:MonSymmFun}
Let $\{a_i\}_{i=1}^d \subset \mathbb{C}$ be a set of complex numbers. Then
$$
\sum_{\substack{I,J \subseteq [d]\\ |I| = k, |J| = l \\ |I \cap J| =m}}
        \Big(\prod_{i\in I} a_i\Big)
       \Big(\prod_{j\in J} a_j\Big) = \binom{k+l-2m}{k-m} M_m(a),
$$
where $M_m(a)$ is the monomial symmetric polynomial associated to the two-column partition $\lambda = (2^{m},1^{k+l-2m})$ in the variables $a_1,...,a_d$.
\end{proposition}

\begin{proof}
    Note that the subsets $(I \cup J) \setminus (I\cap J)$, $I \cap J$ and $[d] \setminus (I \cup J)$ form a partition of $[d]$. Given such a partition we form a sequence of $d$ numbers $\lambda = (\lambda_1,...,\lambda_d)$ where 
$$
\lambda_k = \begin{cases}
0, & \text{if $k \in [d]\setminus (I\cup J)$} \\
1, & \text{if $k \in I \cup J \setminus (I\cap J)$} \\
2, &  \text{if $k \in I \cap J$}.
\end{cases}
$$
It then follows that, for fixed subsets $I,J$, we have 
$$
\Big(\prod_{i\in I} a_i\Big)
       \Big(\prod_{j\in J} a_j\Big) = a_1^{\lambda_1}...a_d^{\lambda_d} = a^\lambda.
$$
If we fix the cardinality of $I \cap J$ to be $m$ it is clear that the associated sequence $\lambda$ will contain a total of $m$ twos, $k+l-2m$ ones and $d+m-k-l$ zeros. Furthermore, as we range over all $I, J$ such that $|I \cap J | = m$ it is clear that we obtain all possible sequences $\lambda$ containing a total of $m$ twos, $k+l-2m$ ones and $d+m-k-l$ zeros. In fact, for a fixed sequence $\lambda$ there are a total of $\binom{k+l-2m}{k-m}$ choices of the subsets $I$ and $J$ that realize $\lambda$. This follows from the fact that a given $\lambda$ fixes the $d+m-k-l$ elements in $[d]\setminus (I \cup J)$ and the $m$ elements in $I \cap J$, since the indices where the sequence is 0 must be in $[d]\setminus (I \cup J)$ while the indices where the sequence is $2$ must be in $I \cap J$. We are therefore free to choose which of the indices where the sequence is $1$ belong to $I \setminus (I \cap J)$. Since $|I|=k$, we are free to choose $k-m$ indices amongst the $k+l-2m$ indices where the sequence is equal to 1. By definition of the monomial symmetric polynomial associated to the partition $\lambda = (2^m,1^{k+l-2m})$ we arrive at the desired expression.
\end{proof}

\subsection{Hahn Polynomials}
\label{sec:HahnPoly} We refer the reader to the monographs \cite{ceccherini1945harmonic,ceccherini1945harmonic} for a very comprehensive exposition about the notion of Gelfand pairs.
Recall that a pair of finite groups $(G,K)$ with $K \leqslant G$ a subgroup of $G$ is called a {\it{Gelfand pair}} if the trivial representation of $K$ induces a multiplicity-free representation of $G$. 
To a Gelfand pair $(G,K)$ one can associate its family of {\it spherical functions}.
\begin{definition}[Spherical functions]
	Let $(G,K)$ be a Gelfand pair. Let ${\rm Ind}_K^G(1) = \oplus_{i = 1 }^s V_i$ be the decomposition of the induced trivial representation into non-isomorphic irreducible representations $V_i$ and let $\chi_i$ denote the character associated to $V_i$. The {\it spherical function} associated to the irreducible representation $V_i$ is the function $\omega_i : G \to \mathbb{C}$ defined by
	$$
		\omega_i(g) = \frac{1}{|K|}\sum_{k \in K} \chi_i(g^{-1}k).
	$$
\end{definition}
In the case of the Gelfand pair $(S_{k+l}, S_k \times S_l)$, explicit formulas for the spherical functions are known. At this point, we insist on the fact that the proof of Theorem \ref{thm:MainTheorem} utilizes explicit formulas for the spherical functions of the Gelfand pair $(S_{k+l}, S_k \times S_l)$, but it would be far more satisfactory to give a conceptual argument on the appearance of the binomial transform which collapses the hahn polynomials to a simple combinatorial coefficient. 
\begin{theorem*}\cite[Theorem 6.2.3, Remark 6.2.4] {ceccherini1945harmonic}
	Let $0\leq h\leq \min(k,l)$ and $\lambda = (k+l-h,h)$. Let $I$ be a $k$-element subset of $[k+l]$. Write $\delta_{I,x}$ for the function on $S_{k+l}$ whose value on $\sigma \in S_{k+l}$ is 1 if $|\sigma(I) \cap I| = k-x$ and 0 otherwise. The spherical function $\omega_h$ associated to the irreducible representation of $S_{k+l}$ indexed by the partition $\lambda$ 
	\begin{equation*}
		\omega_h = \sum_{x=0}^{\min(k,l)} Q_h(x;l,k) \delta_{I,x},
	\end{equation*}
	where the $Q_h(x;l,k)$ are a family of orthogonal polynomials known as {\rm Hahn polynomials}.
\end{theorem*}
Note in particular that given $\sigma \in S_{k+l}$ we have
$$
	\omega_h(\sigma) = Q_h(k-|\sigma(I) \cap I|;l,k).
$$
The Hahn polynomials are defined by
\begin{align*}
	Q_h(x;l,k) = \frac{1}{\binom{k}{h}} \sum_{j=0}^h (-1)^j \frac{\binom{k-h+j}{j}}{\binom{l}{j}} \binom{k-x}{h-j} \binom{x}{j}
\end{align*}
for any $h = 0,\ldots,\min(k,l)$ and $x = 0,\ldots,\min(k,l)$. They are orthogonal polynomials with respect to the hypergeometric distribution
\begin{align*}
	h(x)= \frac{\binom{k}{x}\binom{l}{l-x}}{\binom{k+l}{k}},\quad x \leq \min(k,l).
\end{align*}

and admit the following representation as a ${}_3F_2$ hypergeometric function:
\begin{align*}
	Q_h(x;l,k) = {}_3F_2(\begin{matrix}-h, h-l-k-1, -x \\ -l, -k\end{matrix}; 1).
\end{align*}
See \cite{Karlin1961} for a detailed treatment of the Hahn polynomials and their many interesting properties. 
We will see throughout the proof of Theorem \ref{thm:MainTheorem} that the signed binomial transform of the Hahn polynomials
$$
	\tilde{Q}_h(a)=\sum_{x=0}^k\binom{a}{x}(-1)^{x}Q_h(x;l,k)
$$
appears naturally. We compute it here for the sake of clarity. In what follows, given positive integers $d,n\geq 1$, we denote $(d)_{\bar{n}}= d(d+1)...(d+n-1)$ the \emph{rising factorial} (note the bar on the $n$). The following simple relation holds $(d)_n = (-1)^n(d)_{\bar{n}}$.

\begin{proposition}\label{prop:binomialtransform} 
Let $a,h$ be integers such that $0 \leq a,h \leq k$. Then we have
	\begin{align}
		\tilde{Q}_h(a) & = \frac{(h)_a(l+k+1-h)_a}{(l)_a(k)_a} \text{ if } a \leq h,
	\end{align}
	and $0$ otherwise.
\end{proposition}
\begin{proof}
	From the $_3F_2$ representation of the Hahn polynomials we get
	\begin{align*}
		\tilde{Q}_h(a) & =\sum_{x=0}^k(-1)^{x}\binom{a}{x}{}_3F_2(\begin{matrix}-h, h-l-k-1, -x \\ -l, -k\end{matrix}; 1) \\
		               & =\sum_{x=0}^k(-1)^{x}\binom{a}{x} \sum_{j=0}^h \frac{(-h)_{\bar{j}}(h-l-k-1)_{\bar{j}}(-x)_{\bar{j}}}{(-l)_{\bar{j}}(-k)_{\bar{j}} j!}                      \\
		               & =\sum_{j=0}^h \frac{(-h)_{\bar{j}}(h-l-k-1)_{\bar{j}}}{(-l)_{\bar{j}}(-k)_{\bar{j}} j!} \sum_{x=0}^k(-1)^{x}\binom{a}{x}(-x)_{\bar{j}}                      \\
		               & =\sum_{j=0}^h (-1)^j\frac{(-h)_{\bar{j}}(h-l-k-1)_{\bar{j}}}{(-l)_{\bar{j}}(-k)_{\bar{j}}} \sum_{x=0}^k(-1)^{x}\binom{a}{x}\binom{x}{j}             \\
		               & = \sum_{j=0}^h (-1)^j\frac{(-h)_{\bar{j}}(h-l-k-1)_{\bar{j}}}{(-l)_{\bar{j}}(-k)_{\bar{j}}} \binom{a}{j}\sum_{x=j}^a(-1)^x\binom{a-j}{x-j}          \\
		               & = \sum_{j=0}^h \frac{(-h)_{\bar{j}}(h-l-k-1)_{\bar{j}}}{(-l)_{\bar{j}}(-k)_{\bar{j}}} \binom{a}{j}\delta_{j=a}                                      \\
		               & =\begin{cases}
			                  \frac{(-h)_{\bar{a}}(h-l-k-1)_{\bar{a}}}{(-l)_{\bar{a}}(-k)_{\bar{a}}} & \text{if } a \leq h, \\
			                  0                                      & \text{otherwise.}
		                  \end{cases}                                                      \\
	\end{align*}
Factoring out the negative signs and 
changing the rising factorials to falling factorials gives the result.
\end{proof}
\subsection{Notation} Let $I,J \subseteq \{1,...,d\}$ be two subsets and ${\bf p}: I \to \{1,...,d\}$, ${\bf q}: J \to \{1,...,d\}$ two set maps.
\label{sec:Notation}\begin{itemize}
	\item
	      To distinguish between the two copies of the set $I\cap J$ in $I\sqcup J$, we denote by $(I\cap J)_I$ the one that sits inside $I \subset I\sqcup J$ and $(I\cap J)_J$ the one that sits in $J \subset I\sqcup J$. If $x \in I\cap J$, we denote by $x_I$ (resp. $x_J$) the copy of $x$ in $(I\cap J)_I$ (resp. $(I\cap J)_J$).
    \item The map $\copairpq: I \sqcup J \to \{1,...,d\}$ is the {\it copairing} of ${\bf p}$ and ${\bf q}$. That is, $\copairpq$ maps $x \in I \subset I \sqcup J$ to ${\bf p}(x)$ and $y \in J \subset I \sqcup J$ to ${\bf q}(y)$.
	\item We write $S_{I\sqcup J}$ for the symmetric group of permutations of the set $I\sqcup J$. Given $\sigma \in S_I$ and $\tau \in S_J$ two permutations, we denote $\sigma \oplus \tau \in S_{I}\times S_J \subset S_{I \sqcup J}$ the permutation that acts as $\sigma$ on elements of $I\subset I \sqcup J$ and $\tau$ on elements of $J \subset I \sqcup J$.
	\item We write $\alpha \leq \pi$, with $\alpha$ a permutation of $I\sqcup J$ and $\pi$ a partition of $I \sqcup J$, if the cycles of $\alpha$ yield a partition of $I\sqcup J$ less than $\pi$.
	\item We denote by $c_2(\gamma)$ the number of 2-cycles in the cycle decomposition of a permutation $\gamma \in S_{I \sqcup J}$. 
	\item We write $\varepsilon(\gamma)$ for the signature of a permutation $\gamma$ of $I\sqcup J$. Note that if $\gamma$ is a product of disjoint 2-cycles, then $\varepsilon(\gamma) = (-1)^{c_2(\gamma)}$
    \item We denote by $\iota$ the canonical set map from $I \sqcup J$ to $I \cup J$.
\end{itemize}

\section{Preparatory calculations}
Proving our main result, Theorem \ref{thm:MainTheorem}, requires several steps; we take the first one in this section.
Let $d \geq 1$ be a positive integer and $U = (u_{ij})_{1\leq i,j \leq d}$ a Haar-distributed unitary matrix of size $d$.
Let $I,J \subset [d]$ be two subsets and ${\bf p}, {\bf q}$ two set maps from the sets $I$, resp. $J$, to $[d]$. Set
$
	k = |I|,\, l = |J|$ and $m = |I \cap J|.
$

The objective of this section is to compute the quantity
\begin{equation}
	\label{eqn:OfInterest}
	F({\bf p}, {\bf q}):=\sum_{\sigma \in S_I, \tau \in S_J} \varepsilon (\sigma)\varepsilon(\tau)\mathbb{E}\Big(\prod_{i \in I} u_{i\mathbf{p}(i)}\bar{u}_{\sigma(i)\mathbf{p}(i)}\prod_{j \in J}u_{j\mathbf{q}(j)}\bar{u}_{\tau(j)\mathbf{q}(j)}\Big)
\end{equation}
in terms of $k$, $l$ and $m$, the kernel ${\sf Ker}(\copairpq)$ of the copairing $ \copairpq$ and the Hahn Polynomials introduced in Section \ref{sec:HahnPoly}. The reason for focusing on the quantity above will be apparent in Section \ref{sec:ProofMainTheorem}: it is precisely the quantity that appears when expanding out the second moment $\mathbb{E}\bigl[e_k(AUBU^*)e_l(AUBU^*)\bigr]$. As we will see in Proposition \ref{prop:FIJPQ} below, there is also a dependence on a combinatorial quantity computed on a pair of partitions $({\sf K, {\sf K'}})$ belonging to the set $\setK$. To compute (\ref{eqn:OfInterest}), we will use the Weingarten calculus, see Section \ref{sec:WeinCalc}.
Recall that the kernel of $f\colon I \sqcup J \to \mathcal{X}$ is the partition of $I \sqcup J$ defined by the equivalence relation
\begin{equation*}
	x \sim y \Leftrightarrow f(x) = f(y).
\end{equation*}

\subsection{Computation of $F({\bf p}, {\bf q})$}
\label{sec:ComputationFpq}
We now state one of our two key propositions. Let ${\sf M}$ be the kernel of the set map $\iota$; it is a partition in $\setK$ with type $t({\sf M })=(2^m, 1^{k+l-2m})$. Each element in $I \cap J$ corresponds to a block of size two in ${\sf M}$ and each element in $I \Delta J$ corresponds to a singleton in ${\sf M}$.
\newcommand{\Sr}{S(r, {\sf K})}
\begin{proposition}
\label{prop:FIJPQ}
Under the notations introduced above,
\begin{enumerate}
\item if $\mathbf p$ or $\mathbf q$ is not injective, then
\[
F(\mathbf p,\mathbf q)=0,
\]
\item otherwise,
setting ${\sf K}={\sf Ker}([\mathbf p,\mathbf q])$ and
\[
	S(r, {\sf K}) =
	\sum_{\substack{\eta \leq {\sf M},\, \gamma \leq {\sf K}\\
	|\eta\gamma(I)\cap I|=r}}
	\varepsilon(\eta)\varepsilon(\gamma),
\]
then it holds that (see equation \eqref{eqn:coeffc})

\begin{align}
F(\mathbf p,\mathbf q)
&= k!\,l!\sum_{r=0}^k S(r,{\sf K})
\sum_{h=\max(0,k+l-d)}^{\min(k,l)}
c(2^h_{k+l},d)Q_h(k-r;l,k).
\end{align}

\end{enumerate}
\end{proposition}


The remainder of this section is devoted to the proof of Proposition \ref{prop:FIJPQ}.
\begin{proof}
We prove the two assertions in the statement separately. By using the Weingarten calculus, we expand the right hand-side of \eqref{eqn:OfInterest} as follows:
\begin{equation}
	\label{eqn:FIJPQ}
	F({\bf p}, {\bf q}) = \sum_{\sigma \in S_I}\sum_{\tau \in S_J}\varepsilon(\sigma)\varepsilon(\tau)
	\sum_{\substack{\pi \in \Pi_{\sigma,\tau}\\\gamma\leq {\sf Ker}(\copairpq)}}
	\text{Wg}_d(\pi, \gamma),
\end{equation}
where, for each $\sigma \in S_I$ and $\tau \in S_J$, we have set
$$
	\Pi_{\sigma,\tau}= \{\pi \in S_{I \sqcup J} : \iota(k) = \iota((\sigma\sqcup \tau)(\pi(k))) \text{ for all } k \in I \sqcup J\}.
$$
We prove the first point of Proposition \ref{prop:FIJPQ}. That is, we show that if $\mathbf p: I \rightarrow [d]$ or $\mathbf q: J \rightarrow [d]$ is not injective then
\begin{equation*}
	\sum_{\sigma \in S_I}\sum_{\tau \in S_J}\varepsilon(\sigma)\varepsilon(\tau)
	\sum_{\substack{\pi \in \Pi_{\sigma,\tau}\\\gamma\leq {\sf Ker}(\copairpq)}}
	\text{Wg}_d(\pi, \gamma) = 0.
\end{equation*}

By symmetry, we can just consider the case $\mathbf {p} $ not injective. Therefore, suppose there exists $x \neq x' \in I$ such that $\mathbf p(x) = \mathbf p(x')$.

For a given $\sigma \in S_I$ consider the permutation $\sigma ' = \sigma \circ (xx') \in S_I$. Note that $\varepsilon(\sigma') = -\varepsilon(\sigma)$.

We claim that for all $\tau \in S_J$ we have $\Pi_{\sigma', \tau} = (xx')\Pi_{\sigma, \tau}$. In fact, we have $\pi' \in \Pi_{\sigma', \tau}$ if and only if $\iota(z) = \iota((\sigma'\sqcup\tau)(\pi'(z)))$ for all $z \in I \sqcup J$. Since $\sigma' = \sigma \circ (xx')$ this is equivalent to
$$
		\iota(z) = \iota((\sigma\circ (xx')\sqcup \tau)(\pi'(z))) = \iota((\sigma\sqcup \tau)((x,x')\pi'(z))), \quad \text{for all } z \in I \sqcup J.
$$
Hence, $\pi' \in \Pi_{\sigma', \tau}$ if and only if $(xx')\pi' \in \Pi_{\sigma, \tau}$, which proves the claim.
For $\sigma \in S_I$ we have
\begin{align*}
	\sum_{\tau \in S_J}\varepsilon(\tau)
	\sum_{\substack{\pi' \in \Pi_{\sigma',\tau} \\\gamma\leq {\sf Ker}(\copairpq)}}
	\text{Wg}_d(\pi', \gamma) &= \sum_{\tau \in S_J}\varepsilon(\tau)
	\sum_{\substack{\pi \in \Pi_{\sigma,\tau}   \\\gamma\leq {\sf Ker}(\copairpq)}}
	\text{Wg}_d((xx')\pi, \gamma)               \\
	&=\sum_{\tau \in S_J}\varepsilon(\tau)
	\sum_{\substack{\pi \in \Pi_{\sigma,\tau}   \\\gamma\leq {\sf Ker}(\copairpq)}}
	\text{Wg}_d(\pi, (xx')\gamma)               \\
	&= \sum_{\tau \in S_J}\varepsilon(\tau)
	\sum_{\substack{\pi \in \Pi_{\sigma,\tau}   \\\gamma\leq {\sf Ker}(\copairpq)}}
	\text{Wg}_d(\pi, \gamma),
\end{align*}
where we have used the fact that $\text{Wg}_d$ is a class function in the second equality and the fact that $\gamma \leq {\sf Ker}(\copairpq)$ if and only if $(xx')\gamma \leq {\sf Ker}(\copairpq)$ (because $\mathbf p$ maps $x$ and ${x'}$ to the same element in $[d]$). We have therefore shown that the summands cancel each other out (recall that $\sigma $ and $\sigma'$ have opposite signs), so the sum is equal to 0 and the first point of Proposition \ref{prop:FIJPQ} is proved.
Hereafter, we will assume that both $\mathbf p$ and $\mathbf q$ are injective. In particular, this implies that the copairing $\copairpq$ has a kernel ${\sf Ker}(\copairpq)$ which is a partition in $D(I,J)$ with type $(2^y, 1^{k+l-2y})$, where $y$ is the number of elements in the intersection of the images of $\mathbf p$ and $\mathbf q$.
The dependence on $\sigma$ and $\tau$ of the set $\Pi_{\sigma, \tau}$ (whose dependence on $I$,$J$ is implicit) can be absorbed by the simple change of variables
\begin{equation}
	\label{eqn:changeofvariables}
	\Pi_{\sigma,\tau} \rightarrow  \{\eta \in S_{I\sqcup J} : \eta \leq {\sf M}\}  , \quad \pi \mapsto \eta = (\sigma \oplus \tau) \circ \pi.
\end{equation}
Inserting the formula for the Weingarten function and using the change of variables \eqref{eqn:changeofvariables}, we infer (with the notation of Section \ref{sec:WeinCalc}) that
\begin{align*}
	F({\bf p}, {\bf q}) & =\sum_{\substack{\lambda \vdash (k+l) \\ \ell(\lambda) \leq d}}c(\lambda, d) \sum_{(\sigma, \tau)\in S_I \times S_J}\varepsilon(\sigma \oplus \tau)\sum_{\substack{\eta \leq {\sf M}, \\ \gamma \leq {\sf Ker}(\copairpq)}}\chi^{\lambda}((\sigma \oplus \tau)^{-1} \eta \gamma).
\end{align*}
Given that permutations $\eta \leq {\sf M}$ and $\gamma \leq {\sf Ker}(\copairpq)$ are partial transpositions (they are products of disjoint 2-cycles, hence $\eta = \eta^{-1}$,$\gamma = \gamma^{-1}$), we infer
\begin{align}
 & = \sum_{\substack{\lambda \vdash (k+l)     \\ \ell(\lambda) \leq d}}c({\lambda},d) \sum_{(\sigma, \tau)\in S_I \times S_J}\varepsilon(\sigma \oplus \tau)\sum_{\substack{\eta \leq {\sf M}, \\\gamma \leq {\sf Ker}(\copairpq)}}\chi^{\lambda}(\eta \gamma (\sigma \oplus \tau) ) \notag \\
 & = \sum_{\substack{\lambda \vdash (k+l)     \\ \ell(\lambda) \leq d}}c({\lambda},d) \sum_{\substack{\eta \leq {\sf M}, \\\gamma \leq {\sf Ker}(\copairpq)}} \sum_{(\sigma, \tau)\in S_I \times S_J}\varepsilon(\sigma \oplus \tau)\chi^{\lambda}(\eta \gamma (\sigma \oplus \tau) ) \notag \\
 & =\sum_{\substack{\lambda \vdash (k+l)      \\ \ell(\lambda) \leq d}}c({\lambda},d) \sum_{\substack{\eta \leq {\sf M}, \\\gamma \leq {\sf Ker}(\copairpq)}} \varepsilon(\eta)\varepsilon(\gamma)\sum_{(\sigma, \tau)\in S_I \times S_J}\varepsilon(\eta \gamma (\sigma \oplus \tau))\chi^{\lambda}(\eta \gamma (\sigma \oplus \tau) ) \notag \\
 & = k!l!\sum_{\substack{\lambda \vdash (k+l) \\ \ell(\lambda) \leq d}}c({\lambda},d) \sum_{\substack{\eta \leq {\sf M}, \\\gamma \leq {\sf Ker}(\copairpq)}} \varepsilon(\eta)\varepsilon(\gamma)\quad \underline{\frac{1}{k!l!}\sum_{(\sigma, \tau)\in S_I \times S_J}\chi^{\lambda'}(\eta \gamma (\sigma \oplus \tau) )}, \label{eqn:SpherFuncFpq}
\end{align}
where $\lambda'$ is the conjugate of the partition $\lambda$.

In the underlined part of \eqref{eqn:SpherFuncFpq} we recognize the spherical function of the Gelfand pair $(S_{k+l},S_k\times S_l)$ associated to the partition $\lambda'$ of $k+l$. This is precisely what we mentioned in Section \ref{sec:Intro}. From Section \ref{sec:HahnPoly}, the partitions $\lambda'$ with non-trivial associated spherical functions are the two-row partitions $\lambda'=({k+l-h},h)$ for $0 \leq h \leq k\wedge l$. Taking the conjugate, we find that the partitions $\lambda$ contributing to the sum are
$$
	\lambda = (2^{h}, 1^{k+l-2h}) := 2^{h}_{k+l}, \quad 0 \leq h \leq \text{min}(k,l).
$$
Closed combinatorial formulas for these spherical functions in terms of Hahn polynomials have been recalled in section \ref{sec:HahnPoly}. In particular, when $\lambda = 2^h_{k+l}$, we have that
\begin{align*}
	\frac{1}{k!l!}\sum_{(\sigma, \tau)\in S_I \times S_J}\chi^{\lambda'}(\eta \gamma (\sigma \sqcup \tau) )= Q_h(k-|\eta\gamma(I)\cap I|; l,k)
\end{align*}
with $h = 0, \ldots,\text{min}(k,l)$.
Therefore, plugging this back into \ref{eqn:SpherFuncFpq} we find that 
\begin{align*}
	F({\bf p}, {\bf q})= k!l!\sum_{h=\max(0,k+l-d)}^{\min(k,l)} c(2^h_{k+l},d) \sum_{\substack{\eta\leq {\sf M} \\ \gamma \leq  {\sf Ker}(\copairpq)}} \varepsilon (\eta)\varepsilon(\gamma) Q_h(k-|\eta\gamma(I)\cap I|;l,k),
\end{align*}
where, using the hook-length formula \cite[Theorem 4.2.14]{ceccherini2010representation} and the hook-content formula \cite[Theorem 4.3.3]{ceccherini2010representation}, see also \cite[Lemma 2.7]{campbell2022commutators} :
\begin{equation}
\label{eqn:coefweingarten}
	c(2^h_{k+l},d)= \frac{1}{(k+l)!^2}\frac{\chi^{2^h_{k+l}}({\rm id})^2}{s_{2^h_{k+l}}(1^d)} = \frac{1}{h!}\frac{(k+l+1-2h)}{(k+l-h+1)!} \frac{1}{(d+1)_h(d)_{k+l-h}}.
\end{equation}
The next step is to compute the sum over the permutations $\eta$ and $\gamma$. Recall the quantity
\begin{equation*} 
	S(r, \sK) = \sum_{\substack{\eta \leq {\sf M}, \\ \gamma \leq {\sf K}\\ |\eta\gamma(I)\cap I|=r}}\varepsilon(\eta)\varepsilon(\gamma),\quad  r \in [\min(k,l)].
\end{equation*}
With this notation, we have that
\begin{equation*}
	F({\bf p}, {\bf q})= k!l!\sum_{r=0}^k S(r, {\sf Ker}(\copairpq)) \sum_{h=\max(0,k+l-d)}^{\min(k,l)} c(2^h_{k+l},d) Q_h(k-r; l,k).
\end{equation*}
This proves the second assertion and completes the proof of Proposition \ref{prop:FIJPQ}.
\end{proof}
Recall the definitions of ${\sf C}({\sf M}, {\sf K})$ and $\tildeS{{\sf M}}{\sf K}$ given in Section \ref{sec:MainResult}.
\begin{proposition} \label{prop:FormulePourSx}
	Let ${\sf K} \in \setK$ be a partition of $I\sqcup J$ with type $(2^y,1^{k+l-2y})$  with $y\in [k\wedge l]$. Then, for any $r\in [k]$ we have
	\begin{equation}
		\label{eqn:Sx}
		\Sr= (-1)^{r-k}2^{{\sf C}({\sf M}, {\sf K})}\binom{\tildeS{\sf M}{\sf K}}{k-r}.
	\end{equation}
\end{proposition}

In proving Proposition \ref{prop:FormulePourSx}, we will see that \emph{dramatic cancellations} occur in the double sum defining $\Sr$, thus yielding \eqref{eqn:Sx}.
Before proceeding with the proof, we must introduce some more notation.

Given a permutation $\gamma \leq {\sf K}$ and a subset $S \subset I\sqcup J$, the subset of elements fixed by $\gamma$ in $S$ is denoted $S^\gamma$.

For a partition ${\sf K} \in D(I,J)$, we write $\tilde{\sf K}=\iota^{\star}({\sf K}\vee M)$ for the partition of $I\cup J = \iota(I\sqcup J)$ defined as the pushforward by ${\iota}$ of the join ${\sf K}\vee {\sf M}$. By definition, $x \sim_{\sf \tilde{K}} y$ if and only if there exist $x'\in {\iota}^{-1}(x)$, $y'\in {\iota^{-1}(y)}$ such that $x'\sim_{\sK \vee \sM} y'$.
Since by definition ${\sf Ker}(\iota) = \sM$ and $\sM \prec \sK \vee \sM$, blocks of $\sK$ and $\tilde\sK$ are in bijection. We  denote $B \in {\sf K}\vee \sM$ and $\tilde{B} \in \tilde{K}$ pairs of blocks corresponding under this bijection.

We say that a block
$\tilde{B}\in \tilde{\sf K}$ \emph{contains} a block $P\in {\sf K}$ if $P$ is contained in the corresponding block $B$ in $\sK \vee \sM$.

\subsection{Proof of Proposition \ref{prop:FormulePourSx} : The sum over $\eta$} To prove Proposition \ref{prop:FormulePourSx}, we first compute the sum over $\eta$. This is the content of the next proposition.
\begin{proposition}\label{prop:sumovereta}
		Let $r\in [k]$. Fix $\gamma \leq {\sf K}$ a permutation less than ${\sf K} \in D(I,J)$. Under the condition $\iota(((I\cap J)_I)^{\gamma}) \,\Delta \,\iota(((I\cap J)_J)^{\gamma}) \neq \emptyset$, one has
	$$
		\sum_{\substack{\eta \leq {\sf M}\\ |\eta\gamma(I)\cap I|=r}}\varepsilon(\eta)=0.
	$$
	Otherwise, 
	\begin{align*}
		\sum_{\substack{\eta \leq {\sf M} \\ |\eta\gamma(I)\cap I|=r}}\varepsilon(\eta)=
		(-1)^{r-|I^\gamma|} \binom{|I\cap J|}{r-|(I\backslash I \cap J)^{\gamma}|}.
	\end{align*}
\end{proposition}
\begin{proof}
    We first define a partition of $(I \cap J)_I$ and a partition of $(I \cap J)_J$. Set
    \begin{align*}
		 & E^J \vcentcolon= \{ j \in (I\cap J)_J : (i,j) \in  \gamma \text{ for some } i \in I \},         \\
		 & L^J \vcentcolon= (I\cap J)_J^{\gamma},
	\end{align*}
	so that
    \begin{equation*}
        (I\cap J)_J = E^J\sqcup L^J,
    \end{equation*}
    and
     \begin{align*}
		 & E^I \vcentcolon= \{ i \in (I\cap J)_I : (i,j) \in  \gamma \text{ for some } j \in J \},         \\
		 & L^I \vcentcolon= (I\cap J)_I^{\gamma},
	\end{align*}
    so that 
    \begin{equation*}
        (I\cap J)_I = E^I \sqcup L^I.
    \end{equation*}
   Define ${\sf m} \in S_{I \sqcup J}$ by ${\sf m} = \prod_{x \in I \cap J}(x_I, x_J)$. Now consider the four block partition $P$ of $I \subset I \sqcup J$  given by
    \begin{equation*}
        P = \{E^I, L^I\} \wedge {\sf m}(\{E^J, L^J\}).
    \end{equation*}
    The blocks of the partition $P$ are indexed by pairs $(X,Y)$ where $X \in \{E^I,L^I\}$ and $Y \in \{E^J, L^J\}$ (the pair $(X,Y)$ corresponds to the block $X \cap {\sf m}(Y)$).
    Since $\eta$ ranges over the permutations less than $\sM$, its action is determined by its restriction to $(I\cap J)_I$. This restriction is, in turn, determined by the family $V^{\eta}_{(X,Y)} \subset (X,Y)$ of subsets on which it coincides with  ${\sf m}$ (since ${\sf m}$ is the largest permutation less than ${\sf M}$).
    Consider the partition of $I \subset I \sqcup J$ given by
    $$
    I =(I\backslash I\cap J)^{\gamma} \sqcup (I\backslash I\cap J)\backslash (I\backslash I\cap J)^{\gamma}\sqcup E^I \sqcup L^I.
    $$ We now partition the set $\mathcal{I}\vcentcolon= \eta\gamma(I) \cap I$ by intersecting it with the partition of $I$ defined above. Firstly, we have
	$$
		(I\backslash I\cap J)^{\gamma} \cap \mathcal{I}=(I\backslash I\cap J)^{\gamma},
	$$
	since $\eta$ acts trivially on elements in $I\backslash (I \cap J)$. Elements in $(I\backslash I\cap J)\backslash(I\backslash I\cap J)^{\gamma}$ are sent by $\gamma$ either to $(I\cap J)_J$ or to its complement in $J$. Elements falling in the latter set cannot be in $\mathcal{I}$ since $\eta$ acts trivially on their image by $\gamma$. Hence, elements $x \in \mathcal{I}\cap (I\backslash I\cap J)\backslash(I\backslash I\cap J)^{\gamma}$ fall in the former set and have images through $\gamma$ in the block $E^J$. That is, we have 
    $$
		\gamma(\mathcal{I}\cap(I\backslash I\cap J)\backslash(I\backslash I\cap J)^{\gamma}) \subset E^J.
	$$
    Applying $\eta$, we get that
	\begin{equation*}
		\eta\gamma(\mathcal{I}\cap(I\backslash I\cap J)\backslash(I\backslash I\cap J)^{\gamma}) \subseteq V^{\eta}(L^I,E^J)\sqcup V^{\eta}(E^I,E^J).
	\end{equation*}
    Elements in $E^I$ are mapped by $\gamma$ either to $E^J$ or to $J\setminus (I\cap J)_J$. Elements belonging to the latter set cannot be in $\mathcal{I}$ since $\eta$ acts trivially on their image by $\gamma$. Therefore, $\gamma$ necessarily maps elements in $\mathcal{I}\cap E^I$ to elements in the block $E^J$, which means that
    $$
		\eta\gamma(\mathcal{I}\cap E^I) \subseteq V^{\eta}(L^I,E^J)\sqcup V^{\eta}(E^I,E^J).
	$$
    By definition of the sets $(\cdot, E^J)$ and $V^{\eta}(\cdot,E^J)$, we have
    $$
    V^\eta(L^I,E^J)\subseteq \eta\gamma(\mathcal{I}\cap(I\backslash I\cap J)\backslash(I\backslash I\cap J)^{\gamma})\sqcup \eta\gamma(\mathcal{I}\cap E^I)
    $$
    and
    $$
    V^\eta(E^I,E^J) \subseteq \eta\gamma(\mathcal{I}\cap(I\backslash I\cap J)\backslash(I\backslash I\cap J)^{\gamma})\sqcup \eta\gamma(\mathcal{I}\cap E^I)
    $$
    Putting the above inclusions together, we find that
    $$
  \eta\gamma(\mathcal{I}\cap(I\backslash I\cap J)\backslash(I\backslash I\cap J)^{\gamma})\sqcup \eta\gamma(\mathcal{I}\cap E^I) =   V^\eta(L^I,E^J) \sqcup V^\eta(E^I,E^J).
    $$
    Finally, it is not difficult to see that
	\begin{equation*}
		\mathcal{I}\cap L^I= ((L^I,E^J)\backslash V^{\eta}(L^I,E^J))  \sqcup ((L^I,L^J)\backslash V^{\eta}(L^I,L^J)).
	\end{equation*}
	We therefore obtain 
	\begin{align*}
		\mathcal{I} = & (I\backslash I\cap J)^{\gamma} \sqcup ((L^I,E^J)\backslash V^{\eta}(L^I,E^J))       \sqcup ((L^I,L^J)\backslash V^{\eta}(L^I,L^J))                                                  \sqcup V^{\eta}(L^I,E^J)\sqcup  V^{\eta}(E^I,E^J).
	\end{align*}

	From the above partition we infer that
	\begin{align}
		\label{eqn:conditionone}
		 |V^{\eta}(E^I,E^J
        )|
		&= |\mathcal{I}| + |V^{\eta}(L^I,L^J)| \nonumber
		                                    -  (|(L^I,L^J)| + |(L^I,E^J)| + |(I\setminus (I\cap J))^\gamma|)                                        \nonumber \\
		                                   & = |\mathcal{I}| + |V^{\eta}(L^I,L^J)|-|I^\gamma|
	\end{align}
    Denote by $s$ the number of 2-cycles of $\eta$, so that
	\begin{align}
		\label{eqn:conditiontwo}
		\sum_{(X,Y)}|V^{\eta}(X,Y) | = s
	\end{align}
	From (\ref{eqn:conditionone}), it follows that the number of $\eta \leq {\sf M}$ with $s$ 2-cycles such that $|\mathcal{I}|=r$ is given by
	\begin{align}
	\label{eqn:CountingEta}
			 & \sum_v\binom{|(L^I,L^J)|}{v}
			 \binom{|(E^I, E^J)|}{r + v-|I^{\gamma}|} \binom{|(E^I,L^J)|+|(L^I,E^J)|}{s-r+|I^\gamma|-2v}
	\end{align}
    Observe that
    \begin{equation}
    \label{eqn:SymDiffEquality}
          |(E^I,L^J)| + |(L^I,E^J)| = |L^I|+ |L^J| -2|(L^I,L^J)|= |\iota(L^I)\Delta \iota (L^J)|
    \end{equation}
    
    Similarly, we have
    $$
    |(E^I,E^J)| = |E^I|-|(E^I,L^J)|= |E^I| - (|L^J|-|(L^I,L^J)|)  = |I\cap J| - |\iota(L^I) \cup \iota(L^J)|,
    $$
    where in the last equality we added and subtracted $|L^I|$.
    Finally, it is clear that
    $$
|(L^I,L^J)| = |\iota(L^I) \cap \iota (L^J)|.
    $$
    Using (\ref{eqn:CountingEta}) and the above equalities we conclude that
    \begin{align*}
		 & \sum_{\substack{\eta \leq {\sf M}                                                                                         \\ |\eta\gamma(I)\cap I|=r}}\varepsilon(\eta)= \sum_{s=0}^{|I\cap J|}\sum_v (-1)^{s-r-2v+|I^{\gamma}|}(-1)^{r-|I^{\gamma}|}
		  \binom{|\iota(L^I)\cap \iota(L^J)|}{v}\binom{|\iota(L^I) \Delta \iota(L^J)|}{s-r-2v+|I^{\gamma}|}  \\
		&\hspace{10cm}\times\binom{|I\cap J|-|\iota(L^I)\cup \iota(L^J)|}{r+ v-|I^{\gamma}|} ,
	\end{align*}
	Interchanging the sums over $s$ and $v$, we find
	\begin{align}
    \label{eqn:SumEta}
		 & \sum_{\substack{\eta \leq {\sf M}                                                                                        \nonumber\\ |\eta\gamma(I)\cap I|=r}}\varepsilon(\eta)= (-1)^{r-|I^{\gamma}|}\sum_v \binom{|\iota(L^I)\cap \iota(L^J)|}{v}\binom{|I\cap J|-|\iota(L^I)\cup \iota(L^J)|}{r+ v-|I^{\gamma}|}\\
		 & \hspace{5cm} \sum_{s=0}^{|I\cap J|} (-1)^{s-r-2v+|I^{\gamma}|}\binom{|\iota(L^I) \Delta \iota(L^J)|}{s-r-2v+|I^{\gamma}|}.
    \end{align}
     Under the condition $|\iota(L^I) \Delta \iota(L^J)|=0$, we get
	\begin{align*}
		\sum_{\substack{\eta \leq {\sf M} \\ |\eta\gamma(I)\cap I|=r}}\varepsilon(\eta)&= (-1)^{r-|I^{\gamma}|}\sum_v \binom{|\iota(L^I)\cap \iota(L^J)|}{v}\binom{|I\cap J|-|\iota(L^I)\cup \iota(L^J)|}{r+ v-|I^{\gamma}|}
	\end{align*}
    Suppose instead that $|\iota(L^I) \Delta \iota(L^J)|>0$. Using (\ref{eqn:conditionone}) and (\ref{eqn:SymDiffEquality}) we have 
    $$
    |I\cap J| -|V^\eta(E^I,E^J)|-|V^\eta(L^I,L^J)| \geq |I\cap J|- |(E^I,E^J)|- |(L^I,L^J)|= |\iota(L^I)\Delta \iota(L^J)|,
    $$
    which implies $|I\cap J|\geq |\iota(L^I) \Delta \iota(L^J)|+r+2v-|I^\gamma|$. Therefore, applying the binomial formula to the second sum in (\ref{eqn:SumEta}) yields
    $$
    \sum_{\substack{\eta \leq {\sf M} \\ |\eta\gamma(I)\cap I|=r}}\varepsilon(\eta) = 0.
    $$
    The desired result now follows from the Vandermonde identity upon noticing that
	$$
		|\iota(L^I)\cap \iota(L^J)|-|I^{\gamma}| =
		|\iota(L^I)|-|I^{\gamma}| = |L^I|-|I^{\gamma}| = |(I\backslash I\cap J)^{\gamma}|,
	$$
	under the condition that $\iota(L^I)\Delta \iota(L^J)=\emptyset$.
    
\end{proof}

\subsection{Proof of Proposition \ref{prop:FormulePourSx} : The sum over $\gamma$}
We now proceed to summing over $\gamma$. Proposition \ref{prop:sumovereta} leads us to the following definition.
\newcommand{\Gammas}{{\Gamma({\sf K, M})}}
\begin{definition}
	We define $\Gammas$ as the set of permutations $\gamma \in S_{I\sqcup J}$ such that
	$$\iota(((I\cap J)^{\gamma}_I)) \,\Delta \,\iota(((I\cap J)^{\gamma}_J)) = \emptyset.$$ This is the set of $\gamma$ that contribute to the second sum over $\gamma$ in $\Sr$.
\end{definition}

To compute the sum over $\gamma$, we need to sort the blocks of $\tilde{\sK}$ into types. We first partition $\tilde{\sK}$ into two subsets $\tilde{\sK}^\circ$ and $\tilde{\sK}^\times$, 
where the subset $\tilde{\sK}^{\circ}$ contains all blocks $\tilde{B}$ of $\tilde{\sK}$ such that $B \in \sK\vee \sM$ contains no singleton of $\sK$. We say that a block $\tilde{B}$ is \emph{marked} if it belongs to the subset $\tilde{\sK}^\times$. The blocks of $\tilde{\sK}$ can be further classified into seven types :
\begin{itemize}
	\item ${\sf I}$ : for blocks contained in $I\cap J$,
	\item ${\sf Z}$ : for blocks intersecting $I\backslash I\cap J$ in one point, $J \backslash I \cap J$ in one point and $I\cap J$,
	\item ${\sf L}$ : for blocks intersecting only $I\cap J$ and $J\backslash I\cap J$ in one point,
	\item ${\reflectbox{\sf L}}$ : for blocks intersecting only $I\cap J$ and $I\backslash I\cap J$ in one point,
	\item $\bullet_I,\bullet_J$ : for singletons contained in $I\backslash I \cap J$ (resp. $J\backslash I\cap J$),
	\item $-$ :  for blocks contained in $I\Delta J$.
\end{itemize}
\newcommand{\tI}{{\sf I}}
\newcommand{\tZ}{{\sf Z}}
\newcommand{\tLJ}{{\sf L}}
\newcommand{\tLI}{{\reflectbox{\sf L}}}
\newcommand{\ts}{{\bullet}}
\newcommand{\tP}{{-}}
We say that a block $B$ of $\sK \vee \sM$ is of type $T$ if the corresponding block $\tilde{B}$ of $\tilde{\sK}$ is of type $T$. Given a block $\tilde{B} \in \tilde{\sK}$, we write $\tilde{B}:T$ to indicate that $\tilde{B}$ is of type $T$. Given a type $T$, we denote $T^\times$, resp. $T^\circ$, the subset of marked, resp. unmarked, blocks of type $T$.
In Figure \ref{fig:BlockTypes} below we illustrate the various types of blocks of $\sK \vee \sM$; this should help clarify the notation chosen for the various blocks.
\begin{figure}[h]
\centering
\def\svgwidth{0.5\linewidth}
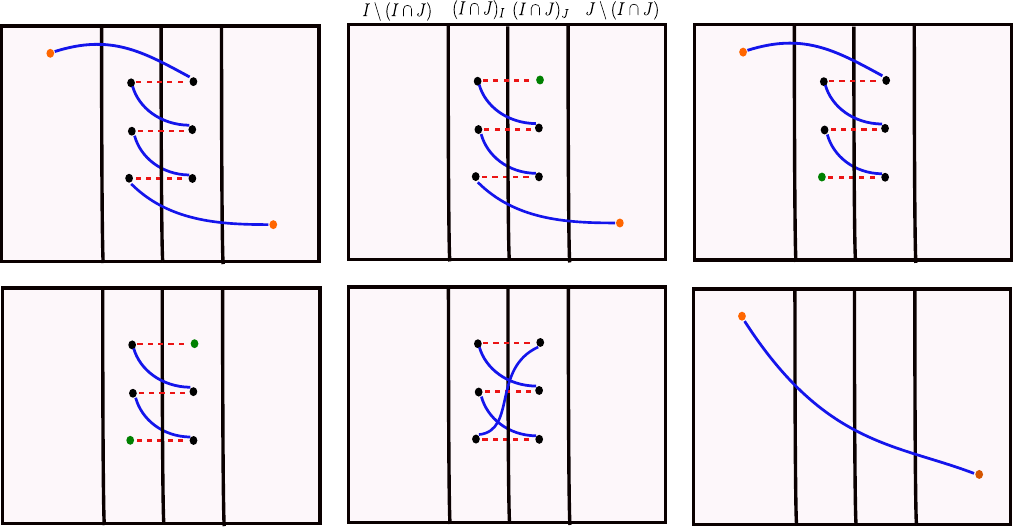
\caption{The boxes represent the set $I\sqcup J$; they are divided into four sections, each section corresponding, in order, to one of the subsets $I \setminus (I \cap J)$, $(I\cap J)_I$, $(I\cap J)_J$ and $J\setminus (I \cap J)$. Edges colored in \emph{blue} connect elements that belong to a block in $\sK$, while dashed edges colored in \emph{red} connect elements that belong to a block in $\sM$. Vertices colored \emph{green} represent elements that are singletons in $\sK$, while vertices colored \emph{orange} represent elements that are singletons in $\sM$. From left to right, the first row depicts blocks of type $\tZ$, $\tLJ$ and $\tLI$. From left to right, the second row depicts blocks of type $\tI^\times$, $\tI^\circ$ and $-$.}
\label{fig:BlockTypes}
\end{figure}
\begin{remark}
\label{rmk:RemarkOnTypesAndCylesSegments}
The fact that blocks of type $\tI^\times$, as depicted above, contain two singletons of $\sK$ is a consequence of Proposition \ref{prop:table}. It is clear from the illustrations above that blocks of $\sK \vee \sM$ of type $\tI^\times$ are precisely the alternating segments of $\sK \vee \sM$ containing singletons of $\sK$ while blocks of type $\tI^\circ$ are precisely the alternating cycles of $\sK \vee \sM$ (so, in particular, we have that $|\tI^\circ|  = {\sf C}(\sK, \sM)$).

\end{remark}
In Proposition \ref{prop:types} we show that these seven types of blocks form a partition of $\tilde{\sK}$. In order to prove Proposition \ref{prop:types}, and subsequent propositions on the types defined above, we introduce the following graph. 
\begin{definition}
	We denote $G_{\sf K}$ the graph whose set of vertices is the set of blocks of ${\sf K} \in \setK$ such that two blocks $B_1, B_2$ are adjacent if there exists $x \in I\cap J$ such that $x \in B_1$ and ${\sf m}(x) \in B_2$ (recall that ${\sf m} \in S_{I \sqcup J}$ is the permutation defined by ${\sf m}=\prod_{x \in {I \cap J}} (x_I,x_J)$).
\end{definition}
\newcommand{\graphK}{{G_{\sf K}}}
\begin{remark}\label{rmk:graphKremarks}
	Each vertex of $\graphK$ has degree at most 2: blocks contained in $I \Delta J$ have degree 0, singletons contained in $I \cap J$ have degree 1, pairs intersecting $I\cap J$ and $I \Delta J$ have degree 1 and pairs contained in $I \cap J$ have degree 2. This immediately implies that the connected components of $\graphK$ are either isolated vertices, cycles or paths. Isolated vertices contain a single degree 0 vertex, cycles  contain only degree 2 vertices while paths contain at least one vertex with degree 1. In particular, cycles are contained in $I \cap J$, since a vertex has degree 2 if and only if its corresponding block is contained in $I \cap J$. We will represent a component $C$ of $\graphK$ via a sequence of vertices $C = (v_0,...,v_k)$ such that $\{v_i,v_{i+1}\}$ is an edge of $\graphK$. Note that each $v_i$ is a block of $\sK$. By a slight abuse of language, we will say that a component $C$ of $\graphK$ intersects a subset $S \subseteq I \sqcup J$ if $C$ contains a vertex that intersects $S$. Given a component $C$ of $\graphK$, we  denote $C^\flat$ the set defined by
	$$
		C^\flat \vcentcolon= v_0 \cup...\cup v_n.
	$$
	This is just the set of all points in $I \sqcup J$ appearing as an  element in one of the vertices of $C$. By definition of the join of two partitions, the set $C^\flat$ is a block of $\sK \vee \sM$ and all blocks of $\sK \vee \sM$ arise as sets $C^\flat$ for some component $C$ of the graph $\graphK$.

\end{remark}

\begin{lemma}\label{lem:componentsGK}
	Let $P$ be a path component of the graph $\graphK$. Then $P$ intersects $I \Delta J$ in at most two points. Furthermore, if $P$ intersects $I \Delta J$ in exactly two points then it must intersect both $I \setminus (I\cap J)$ and $J \setminus (I\cap J)$.
\end{lemma}
\begin{proof}
	Any vertex in the path component $P$ whose corresponding block in ${\sf K}$ intersects $I\Delta J$ has degree 1 by Remark \ref{rmk:graphKremarks}.  Since a path component contains two degree 1 vertices, it follows that $P$ contains at most two vertices whose corresponding blocks intersect $I \Delta J$.
    
    Now suppose that $P = (v_0,...,v_n)$ is a path component that intersects $I\Delta J$ in two points $a\neq b \in I\Delta J$. The vertices associated to the blocks containing $a$ and $b$ have degree 1 so they are necessarily the terminal vertices of $P$. Assume without loss of generality that $a \in v_0$ and $b \in v_n$. Now suppose that $a$ and $b$ are both in $I \setminus (I \cap J)$. This means that $v_0 = \{a,x_J\}$ and $v_n = \{b,y_J\}$ for some $x_J \neq y_J \in (I\cap J)_J$. Since the $n-2$ vertices $v_1,...,v_{n-1}$ are all contained in $I \cap J$ (they all have degree 2), the vertices of $P$ contain a total of $n-2$ elements in $(I \cap J)_I$ and $n$ elements in $(I\cap J)_J$. This is a contradiction, since giving a path component $P$ of $\graphK$ is equivalent to giving a bijection between the sets $P^\flat \cap (I\cap J)_I$ and $P^\flat \cap (I\cap J)_J$ such that each element is sent to its symmetric, which is impossible in this case because the two sets have different cardinalities. A similar argument shows that $a$ and $b$ cannot both belong to $J \setminus (I\cap J)$. We therefore conclude that $P$ necessarily intersects both $I \setminus (I \cap J)$ and $J \setminus (I \cap J)$.
\end{proof}

\begin{proposition} \label{prop:types}
	The seven types of blocks of $\tilde{\sK}$ defined above form a partition of $\tilde{\sK}$. Furthermore, blocks of types $\ts_I,\ts_J$, $\tLJ$ and $\tLI$ are marked, i.e. contained in $\sK^{\times}$, while blocks of type $\tZ$ are not.
\end{proposition}
\begin{proof}
 We begin by showing that the seven types of blocks of $\tilde \sK$ form a partition of $\tilde \sK$. All we need to show is that any block of $\tilde{K}$ that intersects $I \cap J$ and $I\Delta J$ is necessarily a block of type $\tLI$, $\tLJ$ or $\tZ$. Let $\tilde{B}$ be a block of $\tilde{\sK}$ intersecting only $I \cap J$ and $J \setminus (I \cap J)$. The corresponding connected component in the graph $G_K$ is a path component $P$ that intersects $I \cap J$ and $J\setminus (I\cap J)$, so by Lemma \ref{lem:componentsGK} it must intersect $J\setminus (I \cap J) $ in exactly one point (the block containing that point is one of the two terminal vertices of $P$). Therefore, the block $\tilde{B}$ intersects $J\setminus (I \cap J)$ in exactly one point, so it is a block of type $\tLJ$. A similar argument shows that any block of $\tilde{\sK}$ intersecting only $I \cap J$ and $I\setminus (I\cap J)$ is necessarily a block of type $\tI$ (that is, it intersects $I \setminus (I \cap J)$ in exactly one point). Finally, suppose that $\tilde{B}$ is a block that intersects $I\cap J$, $I\setminus (I \cap J)$ and $J \setminus (I \cap J)$. This corresponds to a component $P$ in $G_{\sf K}$ that intersects both $I \setminus (I\cap J)$ and $J \setminus (I\cap J)$. By Remark \ref{rmk:graphKremarks} this component is necessarily a path component and by Lemma \ref{lem:componentsGK} the path component $P$, and therefore the block $\tilde{B}$, intersects $I \setminus (I\cap J)$ in exactly one point and $J \setminus (I \cap J)$ in exactly one point. Therefore, $\tilde{B}$ is a block of type $\tZ$. It is now clear that the seven types of blocks defined above form a partition of $\tilde{K}$.
 
	We now show that blocks of type $\ts_I$, $\ts_J$, $\tLJ$ and $\tLI$ are marked. To do so, we must show that they contain at least one singleton of ${\sf K}$. Blocks of type $\bullet_I$ and $\bullet_J$ are singletons of ${\sf K}$ by definition. Blocks of type ${\sf L}$ intersect only $I \cap J$ and $J \setminus (I \cap J)$. This means that such a block corresponds to a block of ${\sf K} \vee M$ that intersects both $I \cap J$ and $J\setminus (I\cap J)$. By definition, it intersects $J \setminus (I \cap J)$ in exactly one point. Therefore, the corresponding connected component in the graph $G_K$ is a path component $P$ that intersects $J\setminus (I\cap J)$ in exactly one point, and the block containing that point is one of the two terminal vertices of $P$. This means that the other terminal vertex of $P$ is necessarily a singleton, since by Remark \ref{rmk:graphKremarks}   the only degree 1 vertices of $\graphK$ that do not intersect $I \Delta J$ are singletons contained in $I \cap J$.
	An analogous argument applies to $\tLI$.

	We now show that blocks of type $\tZ$ are unmarked, that is, they contain no singletons of \sK. A block of type $\tZ$ corresponds to a block in ${\sf K} \vee M$ that intersects $I \cap J$ and both $I\setminus (I\cap J)$ in one point and $J \setminus (I \cap J)$ in one point, which in turn corresponds to a component $P$ in $G_{\sf K}$ that intersects both $I \setminus (I\cap J)$ in one point and $J \setminus (I\cap J)$ in one point. By Remark \ref{rmk:graphKremarks} this component is necessarily a path component. Since the two vertices of $P$ intersecting $I \Delta  J$ have degree 1, all other vertices must have degree 2, so $P$ does not have any singleton blocks as vertices, which in turn means that the corresponding block in $\tilde{\sK}$ contains no singletons of ${\sK}$.
\end{proof}

\begin{proposition} \label{prop:table}
	Let $\tilde{B} \in \tilde{\sf K}$ be a block of size $n$. In the following table ${\rm type}$ is the type of block $\tilde{B}$ while ${\rm \# pairs}$ is the number of pairs of ${\sf K}$ contained in $B$:
	$$
		\begin{array}{|c|*{7}{c|}}
			\hline
			{\rm{type}}     & \tI^{\times} & \tI^{\circ} & \tZ & \tLJ & \tLI & \ts & \tP \\
			\hline
			{\rm{\# pairs}} & n-1          & n           & n-1 & n-1  & n-1  & 0   & 1   \\
			\hline
		\end{array}
	$$
\end{proposition}
\begin{proof}
	First suppose that $\tilde{B}$ is of type $\ts$. Then $\tilde{B}$ is a singleton contained in $I\Delta J$, which means that the corresponding block $B$ of ${\sf K} \vee M$ is a singleton, so in particular it does not contain any pair of ${\sf K}$.
	\par
	If $\tilde{B}$ is of type $\tP$, then it is contained in $I\Delta J$ and it contains at least two points. This means that the corresponding component $S$ in the graph $\graphK$ contains vertices that are contained in $I\Delta J$. By Remark \ref{rmk:graphKremarks} this is only possible if $S$ contains a single degree 0 vertex, which, as a block of $\sK$ is a pair. This in turn means that the corresponding block $B$ of $\sK \vee \sM$ contains a single pair of $\sK$.
	
	Let us now consider blocks of $\tilde{\sK}$ of types that are fully contained in $I \cap J$: that is, blocks of type $\tI^\circ $ and $\tI^\times$. Let $\tilde{B}$ be a block of $\tilde{K}$ of type ${\sf I}^\circ$ of size $n$. Then $\tilde{B}$ corresponds to a block $B$ of $\sK \vee \sM$ of size $2n$ in $\sK \vee \sM$ (each element of $I\cap J$ is split into its two symmetrics in $(I\cap J)_I$ and $(I\cap J)_J$). By assumption, the block $B$ does not contain any singletons of $\sK$. It then follows that the corresponding component $C$ of the graph $\graphK$ contains only vertices of degree 2 (that is, pairs contained in $I \cap J$), so it is a cycle. The cycle $C$ contains $n$ vertices, and since each such vertex is a pair, we conclude that the block $B$ contains $n$ pairs of $\sK$. If instead $\tilde{B}$ is a block of type $\tI^\times$, then it contains at least one singleton of $\sK$. This means that the corresponding block $B$ in $\sK \vee \sM$ contains a singleton of $\sK$, which in turn means that the corresponding component $P$ in the graph $\graphK$ has at least one vertex with degree 1: that is, $P$ is a path component. In particular, $P$ must contain another vertex of degree $1$, and since $P$ only intersects $I \cap J$, this degree-one vertex must be another singleton contained in $I \cap J$. This means that $B$ contains two singletons of $\sK$, which in turn means that the remaining $2(n-1)$ elements must come from pairs of $\sK$. From this we conclude that $B$ contains $n-1$ pairs of $\sK$.
	
	Now suppose $\tilde{B}$ is a block of type $\tLJ$ or $\tLI$. Then $\tilde{B}$ intersects $I\cap J$ and $I \Delta J$ in at least one point, which means that the corresponding component $P$ in the graph $\graphK$ intersects $I\cap J$ in at least one point and it intersects $I \Delta J$ in exactly one point. This means that $P$ contains at least one vertex of degree $1$ (the one intersecting $I\Delta J$), so $P$ is a path component. In particular, $P$ contains another vertex of degree 1. By Proposition \ref{prop:types}, blocks of type $\tLI$ and $\tLJ$ are always marked, so this other degree 1 vertex is necessarily a singleton contained in $I\cap J$. Consequently, we have that $\tilde{B}$ contains exactly one element in $I\Delta J$ and exactly one singleton of $I \cap J$. If $\tilde{B}$ has size $n$ we then have that the corresponding block $B$ has size $2n-1$. Out of these $2n-1$ elements, we know that one appears as a singleton of $\sK$ while the remaining $2(n-1)$ arise from pairs of $\sK$. We therefore conclude that $B$ contains $n-1$ pairs of $\sK$.
	
	Finally, suppose $\tilde{B}$ is a block of type $\tZ$. Then $\tilde{B}$ intersects $I\Delta J$ in two points. This means that if $\tilde{B}$ has size $n$, the corresponding block $B$ in $\sK \vee \sM$ has size $2n-2$. Since blocks of type $\tZ$ contain no singletons by Proposition \ref{prop:types}, it follows that those $2n-2$ elements all arise from pairs of $\sK$, which means that $B$ contains $n-2$ pairs of $\sK$.
\end{proof}

\begin{lemma}
	Let
	$$
		\mathcal{C}(\tilde{\sK})_\times := \{ f \colon {\tilde{\sK}}\to \{0,1\}: f(\tilde{\sK}^{\times}) = \{0\} \}.
	$$
	Then $\Gammas \simeq \mathcal{C}(\tilde{\sK})_\times$.

\end{lemma}
\begin{proof}
	By Proposition \ref{prop:sumovereta}, the $\gamma$ contributing to $\Sr$ are precisely the $\gamma$ belonging to $\Gammas$; that is, the $\gamma$ such that $\iota((I\cap J)_I^{\gamma}) \,\Delta \,\iota((I\cap J)^{\gamma}_J) = \emptyset$. This last condition is equivalent to the following property: $(P)$ if $x \in (I\cap J)_I$ is fixed by $\gamma$ then so is $x \in (I\cap J)_J$. Since $\gamma \in \Gammas$ is a product of disjoint transpositions such that each transposition exchanges a pair in $\sK$, it is determined by a function $\mathcal{E}_{\gamma} : I\sqcup J \to \{0,1\}$ that is constant on the pairs in $\sK$, equal to ${0}$ on the singletons in $\sK$, and such that a pair $P \in \sK$ appears as a $2$-cycle of $\gamma$ if and only if $\mathcal{E}_{\gamma}(P)=\{1\}$. Property $(P)$ implies that $\mathcal{E}_{\gamma}$ is also constant on the blocks of $M$. Therefore, $\mathcal{E}_{\gamma}$ is constant on the blocks of $\sK \vee \sM$. Reciprocally, any function $\mathcal{E}_{\gamma} \colon I \sqcup J \rightarrow \{0,1\}$ constant on the blocks of $\sK \vee \sM$ and equal to $0$ on the singletons of $\sK$ yields a unique permutation $\gamma \leq \sK$ satisfying property $(P)$ (that is, $\gamma \in \Gammas$). In turn, any such function is uniquely determined by a function $\mathcal{E}_{\gamma}\colon {\tilde \sK} \to \{0,1\}$ equal to $0$ on ${\tilde \sK}^{\times}$. The set of such functions is precisely $\mathcal{C}(\tilde{\sK})_\times$, so we conclude that $\Gammas \simeq \mathcal{C}(\tilde{\sK})_\times$.
\end{proof}
Note that, by Proposition \ref{prop:types}, a function $f \in \mathcal{C}_{\times}(\tilde{\sK})$ is equal to $0$ on blocks of type $\ts$, $\tLI$ and $\tLJ$. 

To keep notation contained, given a type $T$, we denote by $|T|$ the number of blocks with type $T$ in $\tilde{\sK}$. Let $\gamma \in \Gammas$. Given $T\in \{ \tZ, \tI, \tP \}$, write  $T_0$, resp $T_1$, for the subset of blocks of type $T$ that are sent to $0$, resp. 1,  by $\mathcal{E}_{\gamma} \in \mathcal{C}(\tilde{\sK})_\times$.
\begin{lemma}\label{lem:sumsgamma}Let $\gamma \in \Gammas$,
	\begin{align}
		\label{eqn:fixedpointformulasgamma1}
		 & |(I\backslash I \cap J)^{\gamma}|= |\bullet_I| + |\reflectbox{\sf L}| + |{\sf Z}_{0}|+ |-_0|, \hspace{0.1cm} {\rm and}                                                                                                               \\
		\label{eqn:fixedpointformulasgamma2}
		 & |(I \cap J)^{\gamma}|= \sum_{\tilde{B}: {\sf I_0}}|\tilde{B}| + \sum_{\tilde{B}:{\sf Z_0}}(|\tilde{B}|-2) + \sum_{\substack{\tilde{B}:{\sf L}}}(|\tilde{B}|-1) + \sum_{\substack{\tilde{B}:{\scalebox{0.7}{\tLI}}}}(|\tilde{B}|-1).
	\end{align}
	Furthermore, the total number of 2-cycles $c_2(\gamma)$ of $\gamma$ is given by
	\begin{align} \label{eqn:numberofinversionsgamma}
		c_2(\gamma)	= \sum_{\tilde{B}: {\tI_1}}|\tilde{B}| + \sum_{\tilde{B}:{\tZ_1}}(|\tilde{B}| - 1) +|-_1|.
	\end{align}
\end{lemma}
\begin{proof}
	We begin with (\ref{eqn:fixedpointformulasgamma1}). Note that to count the number of fixed points of $\gamma$ on a given set $X$ we must determine the types of blocks of $\tilde{\sf K}$ that intersect $X$ and, among those, determine the ones on which $\mathcal{E}_\gamma$ evaluates to 0. By definition, the only blocks of ${\tilde{\sf K}}$ intersecting $I \setminus (I \cap J)$ on which $\mathcal{E}_\gamma$ evaluates to 0 are those of type $\bullet_I$, $\reflectbox{\sf L}$, ${\sf Z}_0$ and $-_0$. Since blocks of these types each contain exactly one element of $I \setminus (I\cap J)$, we conclude that
	$$
		|(I\backslash I \cap J)^{\gamma}|= |\bullet_I| + |\reflectbox{\sf L}| + |{\sf Z}_{0}|+ |-_0|.
	$$
	We now determine (\ref{eqn:fixedpointformulasgamma2}). The blocks of $\tilde{\sf K}$ that intersect $I \cap J$ and on which $\mathcal{E}_\gamma$ evaluates to 0 on are those of type ${\sf L}$, $\reflectbox{\sf L}$, $Z_0$, and ${\sf I_0}$. A block $\tilde{B}$ of type ${\sf L}$ or $\reflectbox{\sf L}$ contains only one element outside of $I \cap J$, so such a block contributes a $|\tilde{B}|-1$ to $ |(I \cap J)^{\gamma}|$. A block $\tilde{B}$ of type ${\sf Z_0}$ contains exactly two elements outside of $I \cap J$, so such a block contributes a $|\tilde{B}|-2$ to $|(I \cap J)^{\gamma}|$. Finally, blocks of type ${\sf I_0}$ are fully contained in $I \cap J$ so each such block contributes via its full cardinality to $|(I \cap J)^{\gamma}|$. Putting this all together, we find
	$$
		|(I \cap J)^{\gamma}|= \sum_{\tilde{B}\in {\sf I_0}}|\tilde{B}| + \sum_{\tilde{B}:{\sf Z_0}}(|\tilde{B}|-2) + \sum_{\substack{\tilde{B}:{\sf L}}}(|\tilde{B}|-1) + \sum_{\substack{\tilde{B}:{\scalebox{0.7}{\tLI}}}}(|\tilde{B}|-1).
	$$
	Finally, we determine (\ref{eqn:numberofinversionsgamma}). We are counting the number of 2-cycles of $\gamma$, so we restrict our attention to the blocks of $\tilde{\sf K}$ on which $\mathcal{E}_\gamma$ evaluates to 1. Such blocks are necessarily of type ${\sf I_1}$, ${\sf Z_1}$ and $-_1$. The total number of 2-cycles of $\gamma$ coming from any such block $\tilde{B}$ corresponds to the total number of pairs contained in the corresponding block $B$ of $\sK$. From Proposition \ref{prop:table} it follows that
	$$
		c_2(\gamma)= \sum_{\tilde{B}\in{\tI_1}}|\tilde{B}| + \sum_{\tilde{B}\in{\tZ_1}}(|\tilde{B}| - 1) +|-_1|.
	$$
\end{proof}

\subsection{Proof of Proposition \ref{prop:FormulePourSx} : Symmetrization and End of the Proof} Before moving on to the proof of Proposition \ref{prop:FormulePourSx} we prove an interesting and useful property (Proposition \ref{prop:SymmetrySegments}) on the number of alternating segments ${\sf S}(\sK_1,\sK_2)$ of the join $\sK_1 \vee \sK_2$ with $\sK_1, \sK_2$ partitions in $\setK$. To state it, we denote by ${\sf I}_{\sK_2}^{\times}({\sK_1})$ (resp. ${\sf I}_{\sK_1}^{\times}({\sK_2})$) the set of alternating segments of $\sK_1 \vee \sK_2$ containing exactly two elements that belong to singletons in ${\sK}_1$ (resp. ${\sK}_2$). By definition, we have ${\sf S}(\sK_1, \sK_2)=|{\sf I}_{\sK_2}^{\times}({\sK_1})| + |{\sf I}_{\sK_1}^{\times}({\sK_2})| $. 



\begin{proposition}
	\label{prop:SymmetrySegments}
	Let ${\sf K}_1$ and ${\sf K}_2$ be two partitions in $D(I,J)$, of types $t(\sK_1) =(2^{y_1}, 1^{k+l-2y_1})$ and $t(\sK_2) = (2^{y_2}, 1^{k+l-2y_2})$. Then
	\begin{align*}
		y_1  + |{\sf I}^{\times}_{\sK_2} ({\sK_1})| = y_2 + |{\sf I}^{\times}_{\sK_1} ({\sK}_2)|.
	\end{align*}
	In particular,
	\begin{align*}
		\frac{1}{2}(y_1+y_2+{\sf S}({\sf K}_1, {\sf K}_2)) = y_1 + |{\sf I}^{\times}_{\sf K_2} ({\sf K_1})| = y_2 + |{\sf I}^{\times}_{\sf K_1} ({\sf K}_2)|.
	\end{align*}
\end{proposition}
\begin{proof}
    Since any two partitions in $\setK$ of the same type may be mapped to one another via the pushforward of a bijection of $I \sqcup J$ preserving $I$ and $J$, it follows that $\sK_1$ may be pushed forward to $\sM$ via such a bijection. By the invariance of the number of alternating segments under this pushforward, we can assume that $\sK_1 = \sM$.
Let ${\sf K}$ be a partition in $D(I,J)$ of type $(2^y,1^{k+l-2y})$.

The proof is an application of Proposition \ref{prop:table}. We partition the set of pairs of ${\sK}$ according to the blocks of ${\sK \vee \sM}$ and count the number of pairs in each block. This gives
\begin{align*}
	y & = \sum_{\substack{B \in \sK \vee \sM \\ B:\tI^{\circ}}}|B| + \sum_{\substack{B \in \sK \vee \sM \\ B:\tI^{\times}}}(|B|-1)+ \sum_{\substack{B \in \sK \vee \sM \\B:\scalebox{0.7}{\tLI}}}(|B|-1) + \sum_{\substack{B \in \sK \vee \sM \\ B:\tLJ}}(|B|-1) + \sum_{\substack{B \in \sK \vee \sM \\ B:\ts}} 0  \\ & \hspace{1cm}+\sum_{\substack{B \in \sK \vee \sM \\ B:\tP}} 1 + \sum_{\substack{B \in \sK \vee \sM \\ B:\tZ}} (|B|-1) \\
	  & =\sum_{\substack{B \in \sK \vee \sM \\ B:\tI^{\circ}}}|B\cap (I\cap J)| + \sum_{\substack{B \in \sK \vee \sM \\ B:\tI^{\times}}}(|B\cap (I\cap J)|-1)+ \sum_{\substack{B \in \sK \vee \sM \\B:\scalebox{0.7}{\tLI}}}(|B \cap (I\cap J)|) \\ & \hspace{1cm}+ \sum_{\substack{B \in \sK \vee \sM \\ B:\tLJ}}(|B \cap (I\cap J)|)      + |\tP| + \sum_{\substack{B \in \sK \vee \sM \\ B:\tZ}}(|B\cap (I\cap J)|+1)                                                                                          \\
	  & =|I\cap J| - |\tI^{\times}| + |\tZ| + |\tP|.
\end{align*}
We have therefore shown that
\begin{align}
	\label{eqn:relationy}
	y+|{\sf I}^{\times}| = y+|{\sf I}_{\sM}^{\times}({\sK})|=m + |\tZ|+|\tP|,
\end{align}
where the first equality follows from Remark \ref{rmk:RemarkOnTypesAndCylesSegments}.
To conclude the proof, we must now show that
$|\tZ|+|\tP| = |{\sf I}^{\times}_{\sM}({\sK})|$. By definition, a block $B$ of ${\sK}$ such that $\tilde{B}$ is of type ${\sf Z}$ contains exactly two singletons from ${\sf M}$ (that is, two singletons from $I\Delta J$). Therefore, blocks of type $\tZ$ correspond to alternating segments of the graph $\sK \vee \sM$ of cardinality greater than or equal to four containing two elements that are singleton blocks of ${\sM}$. Likewise, blocks of type $\tP$ correspond to alternating segments of $\sK \vee \sM$ of length two such that both elements belong to singleton blocks of ${\sM}$.
Hence, we have that
$
	|-| + |{\sf Z}| = |{\sf I}^{\times}_{\sf M}({\sf K})|
$ and the result is proved.
\end{proof}
Observe that \eqref{eqn:relationy} also implies that $y+|{\sf I}_{\sM}^{\times}({\sK})|$ is less than or equal to $k$, since
\begin{equation}
	\label{eqn:relationtofixedpoints}
	m+ |{\sf Z}|+|-| = k - |\bullet_I|.
\end{equation}
\begin{remark}
Proposition \ref{prop:SymmetrySegments} establishes a symmetry between two partitions $\sK_1$ and $\sK_2$ belonging to $\setK$. This property is precisely what makes the expression of Theorem \ref{thm:MainTheorem} manifestly symmetric in $y$ and $m$.    
\end{remark}
    

We are finally ready to prove Proposition \ref{prop:FormulePourSx}.
\begin{proof}[Proof of Proposition \ref{prop:FormulePourSx}]
We first note that
$$
	\varepsilon(\gamma)(-1)^{|I^{\gamma}|-r}=(-1)^{c_2(\gamma)}(-1)^{|I^{\gamma}|-r} = (-1)^{k-r},
$$
where the second equality follows from the  identity $c_2(\gamma) + |I^\gamma| = k$ (which follows from the fact that $\gamma \leq {\sf K} \in D(k,l)$).
This yields
\begin{align*}
	 & \sum_{\gamma \in \Gammas }\varepsilon(\gamma)(-1)^{r-|I^{\gamma}|}\binom{|I\cap J|}{r-(I\backslash I \cap J)^\gamma} =(-1)^{k-r}\sum_{\gamma \in \Gammas}\binom{|I\cap J|}{r-(I\backslash I \cap J)^\gamma}.
\end{align*}
Since $\Gammas \simeq \mathcal{C}(\tilde{\sK})_\times$, we can sum over $\mathcal{C}(\tilde{\sK})_\times$ instead and, using Lemma \ref{lem:sumsgamma}, we infer that
\begin{align*}
	\sum_{\gamma \in  \Gammas}\binom{|I\cap J|}{r-(I\backslash I \cap J)^\gamma} & = \sum_{\mathcal{E}_{\gamma} \in  \mathcal{C}(\tilde{\sf K})_\times}\binom{|I\cap J|}{r-|\bullet_I|-|\reflectbox{\sf L}|-|{\sf Z}_0|-|-_0|)} \\& =2^{|{\sf I}^{\circ}|}\sum_{l}\binom{|{\sf Z}|+|-|}{l}\binom{|I\cap J|}{r-|\bullet_I|-|\reflectbox{\sf L}|-l} \\
	                                                                                                                              & =2^{|{\sf I}^{\circ}|}\binom{|{\sf Z}|+|-|+|I\cap J|}{-|\bullet_I|-|\reflectbox{\sf L}|+r},
\end{align*}
where the second equality follows from the fact that choosing $\mathcal{E}_\gamma \in \mathcal{C}(\tilde{\sK})_\times $ is equivalent to choosing which blocks of type $\tI^\circ, \tZ$ or $\tP$ it evaluates to 0 on.
We therefore conclude that
\begin{align*}
	\Sr & =2^{|\tI^{\circ}|}(-1)^{k-r}\binom{|\tZ|+|\tP|+|I\cap J|}{-|\ts_I|-|\tLI|+r}                        \\
	              & =2^{|{\tI}^{\circ}|}(-1)^{k-r}\binom{|\tZ|+|\tP|+|I\cap J|}{|{\tZ}|+|\tP|+|I\cap J|+|\ts_I|+|\tLI|-r} \\
	              & =2^{|{\tI}^{\circ}|}(-1)^{k-r}\binom{y+ |{\tI^\times|}}{k-r},
\end{align*}
where in the second equality we used the symmetry of the binomial coefficient and in the third equality we used \eqref{eqn:relationy}. The result follows from Proposition \ref{prop:SymmetrySegments} and Remark \ref{rmk:RemarkOnTypesAndCylesSegments}.
\end{proof}
To move closer to the proof of Theorem \ref{thm:MainTheorem}, we now sum $\Sr$ over all partitions $\sK \in \setK$ with a fixed type. 
\begin{proposition}
	\label{prop:symmetrization}
	Let $y \in [\min(k,l)]$ and $r \in [k]$. Then, under the notation of this section, we have
	\begin{align*}
		\sum_{\substack{{\sK}\in \setK \\ t({\sK})=(2^y,1^{k+l-2y})}} \hspace{-0.75cm}\Sr= \frac{m!(k+l-2m)!}{k!l!}\binom{k+l-2m}{k-m}^{-1}\sum_{\substack{{\sf K_1, \sf K_2} \in \partitionskl\\ {t({\sK}_1)}=(2^m,1^{k+l-2m}) \\  {t({\sK}_2)}=(2^y,1^{k+l-2y})}} (-1)^{r-k}2^{{\sf C}({\sK_1},{\sK_2})} \binom{\tildeS{\sK_1}{\sK_2}}{k-r}
	\end{align*}
\end{proposition}
\begin{proof}
	Let $\sK_2$ be a partition in $\setK$ with type $t(\sK_2) = (2^y,1^{k+l-2y})$. From Proposition \ref{prop:FormulePourSx} and Proposition \ref{prop:SymmetrySegments} we have that
	$$
		S(r,{\sK}_2) = (-1)^{r-k}2^{{\sf C}({\sf K_2}, M)}\binom{y+{\sf I}^{\times}_{{\sM}}({\sK}_2)}{k-r}.
	$$
	It is clear that for any  bijection  $\phi$ of $I\sqcup J$ preserving $I$ and $J$, we have ${\sf C}(\phi^{\star}({\sM}),\phi^{\star}({\sK}_2)) = {\sf C}({\sM},{\sK}_2)$ and ${\sf S}(\phi^{\star}({\sM}),\phi^{\star}({\sK}_2)) = {\sf S}({\sM},{\sK}_2)$. Since any partition ${\sK_1} \in \setK$ with type $(2^m,1^{k+l-2m})$ is the image of the partition $\sM$ under the pushforward of a certain bijection $\phi$ of $I\sqcup J$ preserving $I$ and $J$, we get
	\begin{align*}
		 & \frac{m!(k+l-2m)!}{k!l!}\binom{k+l-2m}{k-m}^{-1}\sum_{\substack{{\sK_1, \sK_2} \in \setK \\ {t({\sK}_1)}=(2^m,1^{k+l-2m}) \\  {t({\sK}_2)}=(2^y,1^{k+l-2y})}} (-1)^{r-k}2^{{\sf C}({\sK_1},{\sK_2})} \binom{\tildeS{{\sK}_1}{{\sK}_2}}{k-r} \\
		 & =\frac{m!(k+l-2m)!}{k!l!}\binom{k+l-2m}{k-m}^{-1}\sum_{\substack{{\sK_1} \in \setK          \\  {t({\sK}_1)}=(2^m,1^{k+l-2m})}}\sum_{\substack{{\sK_2} \in \setK\\  {t({\sK}_2)}=(2^y,1^{k+l-2y})}} (-1)^{r-k}2^{{\sf C}({\sK_1},{\sK_2})}\binom{\tildeS{{\sK}_1}{{\sK}_2}}{k-r} \\
		 & =\frac{m!(k+l-2m)!}{k!l!}\binom{k+l-2m}{k-m}^{-1} \sum_{\substack{{\sK_1} \in \setK        \\  {t({\sK}_1)}=(2^m,1^{k+l-2m})}} \sum_{\substack{{\sK_2} \in \setK\\  {t({\sK}_2)}=(2^y,1^{k+l-2y})}}(-1)^{r-k}2^{{\sf C}({\sM},\phi^{\star}({\sK}_2))} \binom{\tilde{{\sf S}}({\sM},\phi^{\star}({\sK}_2))}{k-r}.
	\end{align*}
	 Since $\phi^{\star}$ is a bijection of the set of partitions ${\sK}_2 \in \setK$ with type $(2^y,1^{k+l-2y})$, a change of variables in the second sum yields

	\begin{align*}
		 & = \frac{m!(k+l-2m)!}{k!l!}\binom{k+l-2m}{k-m}^{-1}|\{ {\sK_1} \in \setK : t(\sK_1)=(2^m, 1^{k+l-2m})\}| \\
		 & \hspace{2cm}\sum_{\substack{{\sK} \in \setK                                                         \\  {t({\sK})}=(2^y,1^{k+l-2y})}}(-1)^{r-k}2^{{\sf C}(\sM,{\sK})} \binom{\tilde{{\sf S}}({\sM},{\sK})}{k-r} \\
		 & = \frac{m!(k-m)!(l-m)!}{k!l!} \binom{k}{m}\binom{l}{m}m! \sum_{\substack{{\sK} \in \setK             \\  {t({\sK})}=(2^y,1^{k+l-2y})}}S(r,{\sK}) \\
		 & = \sum_{\substack{{\sK} \in \setK                                                                    \\  {t({\sK})}=(2^y,1^{k+l-2y})}}S(r,{\sK})
	\end{align*}
	To conclude, choose a bijection $\alpha : I \sqcup J \to \{1,\ldots,k+l\}$, sending $I \subset I\sqcup J$ to $\{1,\ldots k\}$ and $J \subset I \sqcup J$ to $\{k+1,\ldots,k+l\}$. Then the pushforward $\alpha^{\star}$ yields a bijection from $\setK$ to $\partitionskl$ and the result follows.
\end{proof}

\section{Proof of Theorem \ref{thm:MainTheorem}}
\label{sec:ProofMainTheorem}
\begin{proof}[Proof of Theorem \ref{thm:MainTheorem}]
From now on we let $W = AUBU^*$. We have the well-known formula
$$
	e_k(W) = \sum_{\substack{I \subseteq [d]\\ |I| = k}} \text{det}(W(I,I)),
$$
where $W(I,J)$ denotes the submatrix of $W$ whose rows are indexed by $I\subseteq [d]$ and whose columns are indexed by $J \subseteq [d]$. This means that
\begin{align*}
	e_k(W)e_l(W) & = \Big(\sum_{\substack{I \subseteq [d] \\ |I| = k}} \text{det}(W(I,I))\Big)\Big(\sum_{\substack{J \subseteq [d]\\ |J| = l}} \text{det}(W(J,J))\Big) \\
	                             & = \sum_{\substack{I \subseteq [d]      \\ |I| = k}} \sum_{\substack{J \subseteq [d]\\ |J| = l}} \text{det}(W(I,I))\text{det}(W(J,J)).
\end{align*}
Note that
\begin{align*}
	\det(W(I,I))
	= \sum_{\sigma \in S_I} \varepsilon(\sigma) \prod_{i\in I}\Big(a_{i}\sum_{p=1}^d u_{ip}b_p\bar{u}_{\sigma(i)p}\Big)
	= \Big(\prod_{i\in I}a_{i} \Big)\sum_{\sigma \in S_I} \varepsilon(\sigma) \prod_{i \in I}\Big(\sum_{p=1}^d u_{ip}b_p\bar{u}_{\sigma(i)p}\Big).
\end{align*}
From this we deduce that
\begin{align*}
	\det(W(I,I))\det(W(J,J))
	 & = \Big(\prod_{i\in I}a_{i} \Big)\sum_{\sigma \in S_I} \varepsilon(\sigma) \prod_{i \in I}\Big(\sum_{p=1}^d u_{ip}b_p\bar{u}_{\sigma(i)p}\Big)                                                                                         \\
	 & \hspace{2cm}\Big(\prod_{j\in J}a_{j} \Big)\sum_{\tau \in S_J} \varepsilon(\tau) \prod_{j \in J}\Big(\sum_{q=1}^d u_{jq}b_q\bar{u}_{\tau(j)q}\Big)                                                                              \\
	 & = \Big(\prod_{i\in I}a_{i}\Big)\Big(\prod_{j\in J}a_{j}\Big)\sum_{\sigma \in S_I} \sum_{\tau \in S_J} \varepsilon(\sigma)\varepsilon(\tau)                                                                                            \\
	 & \hspace{2cm} \Big(\prod_{i \in I}\Big(\sum_{p=1}^d u_{ip}b_p\bar{u}_{\sigma(i)p}\Big)\Big)(\prod_{j \in J}\Big(\sum_{q=1}^d u_{jq}b_q\bar{u}_{\tau(j)q})\Big)                                                     \\
	 & = \Big(\prod_{i\in I}a_{i}\Big)(\prod_{j\in J}a_{j})\sum_{\mathbf{p}: I \to [d]}\sum_{\mathbf{q}: J \to [d]}(\prod_{i \in I}b_{\mathbf{p}(i)})\Big(\prod_{j \in J}b_{\mathbf{q}(j)}\Big)                        \\
	 & \hspace{2cm} \sum_{\sigma \in S_I, \tau \in S_J} \varepsilon(\sigma)\varepsilon(\tau) (\prod_{i \in I} u_{i\mathbf{p}(i)}\bar{u}_{\sigma(i)\mathbf{p}(i)}\prod_{j \in J}u_{j\mathbf{q}(j)}\bar{u}_{\tau(j)\mathbf{q}(j)}),
\end{align*}
where in the final equality we switched sums and products. Putting everything together and taking the expectation we find
\begin{align*}
	 & \mathbb{E}\bigl[e_k(W) e_{l}(W)\bigr] \nonumber                                                                                                                                                                    \\
	 & \hspace{2cm}= \sum_{\substack{I \subseteq [d]                                                                                                                                                                                                    \\ |I| = k}} \sum_{\substack{J \subseteq [d]\\ |J| = l}} (\prod_{i\in I}a_{i})(\prod_{j\in J}a_{j}) \sum_{\mathbf{p}: I \to [d]}\sum_{\mathbf{q}: J \to [d]}(\prod_{i \in I}b_{\mathbf{p}(i)})(\prod_{j \in J}b_{\mathbf{q}(j)}) \nonumber\\
	 & \hspace{3cm} \sum_{\sigma \in S_I, \tau \in S_J} \varepsilon(\sigma)\varepsilon(\tau)\, \mathbb{E}(\prod_{i \in I} u_{i\mathbf{p}(i)}\bar{u}_{\sigma(i)\mathbf{p}(i)}\prod_{j \in J}u_{j\mathbf{q}(j)}\bar{u}_{\tau(j)\mathbf{q}(j)}).
\end{align*}

Using (\ref{eqn:OfInterest}) and Proposition \ref{prop:FIJPQ} the above expression can be rewritten as

\begin{equation}\label{eqn:productofes}
	= \sum_{\substack{I,J \subseteq [d]\\ |I| = k, |J| = l }} (\prod_{i\in I}a_{i})(\prod_{j\in J}a_{j})\sum_{\substack{\mathbf{p}: I \to [d], \mathbf{q}: J \to [d]\\ \text{injective} }} (\prod_{i \in I}b_{\mathbf{p}(i)})(\prod_{j \in J}b_{\mathbf{q}(j)}) F(\mathbf{p}, \mathbf{q}).
\end{equation}
We now fix $m = |I \cap J|$ and $y = |{\rm Im}({\bf p}) \cap {\rm Im}({\bf q})|$. Note that requiring  $y = |{\rm Im}({\bf p}) \cap {\rm Im}({\bf q})|$ is equivalent to requiring that $t({\sf Ker}(\copairpq)) = (2^y,1^{k+l-2y})$. Since $F({\bf p}, {\bf q})$ depends only on ${\sf Ker}({\bf p}, {\bf q})$, as shown in Proposition \ref{prop:FIJPQ}, write $F({\sf K})$ for this common value. Then (\ref{eqn:productofes}) becomes
\begin{equation}\label{eqn:productofes2}
	= \sum_{m,y=0}^{\text{min}(k,l)}\sum_{\substack{I,J \subseteq [d]\\ |I| = k, |J| = l \\ |I \cap J| =m}}  (\prod_{i\in I}a_{i})(\prod_{j\in J}a_{j})
	\sum_{\substack {{\sK}\in \setK \\ t({\sK}) = (2^y,1^{k+l-2y})}} F({\sK})\sum_{\substack{\mathbf{p}: I \to [d], \mathbf{q}: J \to [d]\\ \text{injective} \\ {\sf Ker}(\copairpq)={\sK}}}  (\prod_{i \in I}b_{\mathbf{p}(i)})(\prod_{j \in J}b_{\mathbf{q}(j)}) .
\end{equation}
Using an argument similar to the one used in the proof of Proposition \ref{prop:MonSymmFun},
we have
$$
	\sum_{\substack{\mathbf{p}: I \to [d], \mathbf{q}: J \to [d]\\ \text{injective} \\ {\sf Ker}(\copairpq)={\sK}}}  (\prod_{i \in I}b_{\mathbf{p}(i)})(\prod_{j \in J}b_{\mathbf{q}(j)}) = {y!(k+l-2y)!}M_y(b),
$$
where $t({\sf K}) = (2^y,1^{k+l-2y})$.
By Proposition \ref{prop:FIJPQ} we have
\begin{equation}
\label{eqn:Fdkl}
	\sum_{\substack {{\sf K}\in D(I,J) \\ t({\sf K}) = (2^y,1^{k+l-2y})}} F({\sf K}) = \sum_{\substack {{\sf K}\in D(I,J) \\ t({\sf K}) = (2^y,1^{k+l-2y})}} \sum_{r = 0}^k S(r,{\sf K}) F_{d,k,l}(r).
\end{equation}

where
$$
	F_{d,k,l}(r)= \sum_{h=\max(0,k+l-d)}^{\min(k,l)} c(2^h_{k+l},d) Q_h(k-r; l,k).
$$

Switching the two sums and using Proposition \ref{prop:symmetrization}, equation (\ref{eqn:Fdkl}) reduces to
\begin{align*}
	= \sum_{r=0}^k F_{d,k,l}(r) \frac{m!(k+l-2m)!}{k!l!}\binom{k+l-2m}{k-m}^{-1}\sum_{\substack{{\sf K_1, \sf K_2} \in D(k,l) \\ {t({\sf K}_1)}=(2^m,1^{k+l-2m}) \\  {t({\sf K}_2)}=(2^y,1^{k+l-2y})}}(-1)^{r-k} 2^{{\sf C}({\sf K_1},{\sf K_2})} \binom{\tilde{\sf S}({\sf K}_1, {\sf K}_2)}{k-r}.
\end{align*}
Plugging everything back into (\ref{eqn:productofes2}) we find
\begin{align}
	 & \mathbb{E}\bigl(e_k(W) e_{l}(W)\bigr) \nonumber                                                                   \\
	 & \hspace{2cm}= \sum_{m,y=0}^{\min(k,l)}  {y!(k+l-2y)!}M_y(b) \sum_{r=0}^k F_{d,k,l}(r) \frac{m!(k+l-2m)!}{k!l!} \binom{k+l-2m}{k-m}^{-1} \nonumber \\
	 & \hspace{2cm} \sum_{\substack{\mathsf{K}_1,\mathsf{K}_2 \in D(k,l)                                                                               \\ t(\mathsf{K}_1)=(2^m,1^{k+l-2m})\\ t(\mathsf{K}_2)=(2^y,1^{k+l-2y})}} (-1)^{r-k}2^{{\sf C}(\mathsf{K}_1,\mathsf{K}_2)} \binom{\tilde{\mathsf{S}}(\mathsf{K}_1, \mathsf{K}_2)}{k-r} \sum_{\substack{I,J \subseteq [d]\\ |I| = k, |J| = l \\ |I \cap J| =m}} (\prod_{i\in I}a_{i})(\prod_{j\in J}a_{j}). \label{eqn:symmetriceqnes}
\end{align}
From Proposition \ref{prop:MonSymmFun} we have that
$$
	\sum_{\substack{I,J \subseteq [d]\\ |I| = k, |J| = l \\ |I \cap J| =m}}  (\prod_{i\in I}a_{i})(\prod_{j\in J}a_{j})= \binom{k+l-2m}{k-m}M_m(a).
$$

Therefore, plugging the above expression back into (\ref{eqn:symmetriceqnes}), simplifying the binomial coefficients and switching the sums we arrive at 
\begin{align*}
	 & \mathbb{E}[e_k(W)e_l(W)]                            \\
	 & \hspace{2cm}= \sum_{m,y=0}^{\text{min}(k,l)}  {m!y!(k+l-2y)!(k+l-2m)!}M_m(a)M_y(b)  \\
	 & \hspace{2cm} \sum_{\substack{\sf K_1, \sf K_2                                     \\ t({\sf K}_1)=(2^m,1^{k+l-2m}) \\ t({\sf K}_2)=(2^y,1^{k+l-2y})}} \sum_{r=0}^k (-1)^{r-k}2^{{\sf C}({\sf K_1},{\sf K}_2)} \binom{\tildeS{\sK_1 }{\sK_2}}{k-r} F_{d,k,l}(r).
\end{align*}

Inserting back the expression of $F_{d,k,l}(r)$ and of $c(2_{k+l}^h,d)$ and using Proposition \ref{prop:binomialtransform} to compute the sum over $r$, we obtain

\begin{align*}
	 & \mathbb{E}[e_k(W)e_l(W)]          \\
	 & \hspace{1cm}=  \sum_{y,m} {m!y!(k+l-2y)!(k+l-2m)!}M_m(a)M_y(b) \\
	 & \hspace{1cm} \sum_{\substack{\sf K_1, \sf K_2                   \\ {t({\sf K}_1)}=(2^m,1^{k+l-2m}) \\ {t({\sf K}_2)}=(2^y,1^{k+l-2y})}} \hspace{-0.5cm} 2^{{\sf C}({\sf K_1},{\sf K_2})} \hspace{-0.5cm} \sum_{h = {\rm max}(0,k+l-d, {\tilde{\sf S}}(\sK_1,\sK_2))}^{{\rm min}(k,l)}  \frac{1}{h!}\frac{(k+l+1-2h)}{(k+l-h+1)!}\frac{(h)_{\tilde{\sf S}({\sf K}_1,{\sf K}_{2})}(l+k+1-h)_{\tilde{\sf S}({\sf K}_1,{\sf K}_{2})}}{(l)_{\tilde{\sf S}({\sf K}_1,{\sf K}_{2})}(k)_{\tilde{\sf S}({\sf K}_1,{\sf K}_{2})}} \\
	 & \hspace{2cm}\frac{1}{(d+1)_h(d)_{k+l-h}}
\end{align*}

Set $a=\tildeS{\sK_1}{\sK_2}$ for brevity. We now rearrange the expression in the third sum above into 

\begin{align*}
\frac{1}{h!}\frac{(k+l+1-2h)}{(k+l-h+1)!}\frac{(h)_{a}(l+k+1-h)_{a}}{(l)_{a}(k)_{a}} &= \frac{1}{l!k!}(k+l+1-2h)\frac{(l-a)!(k-a)!}{(h-a)!(k+l+1-h-a)!} \\
&=\frac{(l-a)!(k-a)!}{k!l!}\frac{(k+l+1-2h)}{(k+l+1-2a)!}\binom{k+l+1-2a}{h-a} \\
&=\frac{1}{k!l!} \frac{(k+l+1-2h)}{\binom{k+l-2a}{k-a}(k+l+1-2a)}\binom{k+l+1-2a}{h-a}
\end{align*}
From the absorption identity for binomial coefficients, we get 
\begin{align*}
(k+l+1-2h)\binom{k+l+1-2a}{h-a} =(k+l+1-2a)(\binom{k+l-2a}{h-a}-\binom{k+l-2a}{h-a-1}).
\end{align*}
Hence, we have
\begin{align*}
\frac{1}{h!}\frac{(k+l+1-2h)}{(k+l-h+1)!}\frac{(h)_{a}(l+k+1-h)_{a}}{(l)_{a}(k)_{a}} = \frac{1}{k!l!}\frac{1}{\binom{k+l-2a}{k-a}}(\binom{k+l-2a}{h-a}-\binom{k+l-2a}{h-a-1})
\end{align*}
Upon replacing 
$$
k+l-2a = (\sK_1,\sK_2)[k+l], \quad k-a = (\sK_1,\sK_2)[k]
$$
and introducing the new summation index $h' = h -\tilde{\sf S}(\sK_1,\sK_2)$ (which we rename $h$) we obtain the desired expression.
\end{proof}

\bibliographystyle{amsalpha}
\bibliography{references}

\end{document}